\documentclass[11pt]{article}

\usepackage[T1]{fontenc}
\usepackage[a4paper,top=0.86in,bottom=0.88in,left=0.90in,right=0.90in]{geometry}
\usepackage{amsmath,amssymb,amsthm,mathtools}
\usepackage{enumitem}
\usepackage{microtype}
\usepackage[hidelinks]{hyperref}

\usepackage[nameinlink,noabbrev]{cleveref}

\hypersetup{
  pdftitle={Exact random covers of metric trees: balanced rounding, duality, and sharp thresholds},
  pdfauthor={Qi Wu and Yong Lu},
  pdfsubject={Random ball covers of finite metric trees},
  pdfkeywords={graph burning, metric tree, random cover, balanced matrix, set-covering polyhedron, fractional duality}
}

\allowdisplaybreaks
\setlist[enumerate]{leftmargin=2.3em,itemsep=0pt,topsep=2pt}

\usepackage{aliascnt}
\newtheorem{theorem}{Theorem}[section]
\newaliascnt{lemma}{theorem}
\newtheorem{lemma}[lemma]{Lemma}
\aliascntresetthe{lemma}
\newaliascnt{proposition}{theorem}
\newtheorem{proposition}[proposition]{Proposition}
\aliascntresetthe{proposition}
\newaliascnt{corollary}{theorem}
\newtheorem{corollary}[corollary]{Corollary}
\aliascntresetthe{corollary}
\newaliascnt{conjecture}{theorem}
\newtheorem{conjecture}[conjecture]{Conjecture}
\aliascntresetthe{conjecture}
\theoremstyle{definition}
\newaliascnt{definition}{theorem}
\newtheorem{definition}[definition]{Definition}
\aliascntresetthe{definition}
\theoremstyle{remark}
\newaliascnt{remark}{theorem}
\newtheorem{remark}[remark]{Remark}
\aliascntresetthe{remark}

\crefname{theorem}{Theorem}{Theorems}
\crefname{lemma}{Lemma}{Lemmas}
\crefname{proposition}{Proposition}{Propositions}
\crefname{corollary}{Corollary}{Corollaries}
\crefname{conjecture}{Conjecture}{Conjectures}
\crefname{definition}{Definition}{Definitions}
\crefname{remark}{Remark}{Remarks}
\crefname{section}{Section}{Sections}
\crefname{subsection}{Subsection}{Subsections}
\Crefname{theorem}{Theorem}{Theorems}
\Crefname{lemma}{Lemma}{Lemmas}
\Crefname{proposition}{Proposition}{Propositions}
\Crefname{corollary}{Corollary}{Corollaries}
\Crefname{conjecture}{Conjecture}{Conjectures}
\Crefname{definition}{Definition}{Definitions}
\Crefname{remark}{Remark}{Remarks}
\Crefname{section}{Section}{Sections}
\Crefname{subsection}{Subsection}{Subsections}

\newcommand{\R}{\mathbb R}
\newcommand{\C}{\mathcal C}
\newcommand{\E}{\mathbb E}
\newcommand{\one}{\mathbf 1}
\newcommand{\Nf}{\mathcal N_{\!f}}
\newcommand{\diamT}{\operatorname{diam}}
\newcommand{\cl}{\operatorname{cl}}
\newcommand{\Lam}[2]{\Lambda_{#1,#2}}
\newcommand{\wgt}{\mathrm w}

\title{Exact random covers of metric trees: balanced rounding, duality, and sharp thresholds\thanks{This work is supported by the National Natural Science Foundation of China (Nos.~12371348 and 12201258) and the High-Quality Science and Technology Cultivation Project of Jiangsu Normal University (No.~JSNUGZL2026069).}}
\author{Qi Wu, Yong Lu\thanks{Corresponding author.}\\[2pt]
\small School of Mathematics and Statistics, Jiangsu Normal University,\\[-1pt]
\small Xuzhou, Jiangsu 221116, People's Republic of China\\[-1pt]
\small Emails:~\texttt{wuqimath@163.com, luyong@jsnu.edu.cn}}
\date{}

\begin{document}
\maketitle

\begin{abstract}
Norin and Turcotte's asymptotically sharp bound for graph burning  [J. Combin. Theory Ser. B 168 (2024), 208--235] led them to an exact random-cover conjecture for finite metric trees. Let $U[0,r]$ be the uniform probability measure on $[0,r]$. They conjectured that every finite metric tree $T$ of length $L\ge2r$ admits a probability measure on $0$-good ball covers whose expected radius measure is at most $(L/r)U[0,r]$. We prove the conjecture for every finite metric tree.

We recast the bootstrapping calculation of Norin and Turcotte as a zero-error replacement certificate. The resulting local scale reduction, together with a three-piece decomposition and a macro-recursion, produces a fractional marked-ball cover with the exact radius budget. We then pass from the fractional cover to random finite covers by a compact rounding argument. For metric-tree balls, Tamir's balancedness theorem and standard balanced-matrix ideality provide the finite-dimensional integrality input.

We also prove an arbitrary-budget duality criterion. If $0<R\le L$ and $\beta$ is a finite positive Borel measure on $[0,R]$, then $\beta$ dominates the expected radius measure of a random $0$-good cover if and only if $\sigma(T)\le\int_{[0,R]}\max_{v\in T}\sigma(B_T(v,s))\,d\beta(s)$ for every finite positive Borel measure $\sigma$ on $T$; it is enough to test finite atomic measures. We use this criterion to extend the uniform range to every $r\le L-\operatorname{diam}(T)/2$, determine the exact range for equal-arm metric stars, and derive deterministic bounds, interval rigidity, and a diameter-defect stability estimate.
\end{abstract}



\section{Introduction}

Graph burning is a discrete-time model for the spread of information or contagion through a network. At each round one new source is chosen, while every earlier fire spreads across one edge. For a vertex $x$ of a graph $G$, let $B_G(x,t)$ be the set of vertices at distance at most $t$ from $x$. A sequence $x_1,\ldots,x_k$ burns $G$ in $k$ rounds if and only if $V(G)=\bigcup_{i=1}^k B_G(x_i,k-i)$. The burning number $b(G)$ is the least such $k$.

An equivalent transmission process on the hypercube appeared in the work of Alon~\cite{Alon}; Norin and Turcotte~\cite[Section~1]{NorinTurcotte} explain this connection. Bonato et al.~\cite{BonatoJanssenRoshanbin2014,BonatoJanssenRoshanbin2016} later introduced the term graph burning and studied the process as a model of social contagion. Bonato et al.~\cite{BonatoJanssenRoshanbin2016} proved $b(P_n)=\lceil\sqrt n\rceil$, and they~\cite[Corollary~2.5]{BonatoJanssenRoshanbin2016} also showed that, for every connected graph $G$, $b(G)=\min\{b(T):T\text{ is a spanning tree of }G\}$. These facts led to the central conjecture of the subject.

\begin{conjecture}[Burning Number Conjecture~\cite{BonatoJanssenRoshanbin2016}]\label{conj:burning}
If $G$ is a connected graph on $n$ vertices, then $b(G)\le\lceil\sqrt n\rceil$.
\end{conjecture}

Paths show that the bound in \cref{conj:burning} would be best possible. The general upper bound has been improved several times. Bonato et al.~\cite{BonatoJanssenRoshanbin2016} proved $b(G)\le2\lceil\sqrt n\rceil-1$. Bessy et al.~\cite{BessyEtAl} improved this to $b(G)\le\sqrt{12n/7}+3$. Land and Lu~\cite{LandLu} then obtained $b(G)\le\lceil(\sqrt{24n+33}-3)/4\rceil$. Bastide et al.~\cite{BastideEtAl} later proved $b(G)\le\sqrt{4n/3}+1$.

The conjectured bound is known for several graph classes. Bonato and Lidbetter~\cite{BonatoLidbetter} and Tan and Teh~\cite{TanTeh} studied spiders and path forests, while Das et al.~\cite{DasEtAl} studied spiders. Hiller et al.~\cite{HillerKosterTriesch} treated $p$-caterpillars, while Liu et al.~\cite{LiuHuHu} proved the conjectured bound for caterpillars. Murakami~\cite{Murakami} proved the conjecture for trees without degree-two vertices. Ning et al.~\cite{NingJinZhang} later proved it for every tree of order $n$ with at most $\lfloor\sqrt{n-1}\rfloor$ degree-two vertices. Omar and Rohilla~\cite{OmarRohilla} obtained further results for several graph classes. Bessy et al.~\cite{BessyHard} proved that the decision problem is NP-complete for acyclic graphs of maximum degree three, spiders, and path forests. Mitsche et al.~\cite{MitschePralatRoshanbin} studied probabilistic forms of graph burning. Bonato~\cite{BonatoSurvey} gave a survey of the subject.

Norin and Turcotte~\cite{NorinTurcotte} proved the asymptotically sharp general bound $b(G)\le(1+o(1))\sqrt n$ for every connected graph of order $n$. Their proof starts with the spanning-tree reduction and then passes from a finite tree to a metric tree. This removes integer scaling from the covering problem and allows the radii to be chosen at random.

We recall the metric setting. Let $T$ be a finite metric tree, and let $|T|$ be its total edge length. A tuple $\mathbf r=(r_1,\ldots,r_m)$ is a cover of $T$ if suitable centers $v_1,\ldots,v_m\in T$ satisfy $T\subseteq\bigcup_i B_T(v_i,r_i)$. It is $\ell$-good if $\sum_i r_i\le |T|+\ell$, and $\C(T,\ell)$ denotes the set of all such covers. If $\nu$ is a probability measure on covers, its expected radius measure is defined by $E_\nu(A)=\int |\{i:r_i\in A\}|\,d\nu(\mathbf r)$ for every Borel set $A\subseteq\R_{\ge0}$. For $0\le a<b$, let $U[a,b]$ be the uniform probability measure on $[a,b]$.

Norin and Turcotte~\cite[Theorem~4.1]{NorinTurcotte} proved that, for $\varepsilon,r>0$ and $|T|\ge24\varepsilon^{-1}r$, there is a probability measure $\nu$ on $\C(T,r)$ such that $E_\nu\le(1+\varepsilon)(|T|/r)U[0,r]$. This is the main continuous statement used in their proof of the asymptotic bound. It has three losses: the factor $1+\varepsilon$, the allowance $\sum_i r_i\le |T|+r$, and the lower bound $|T|\ge24\varepsilon^{-1}r$.

The exact target $\frac{|T|}{r}U[0,r]$ is natural. Its total mass is $|T|/r$, and its first moment is $|T|/2$. The latter value is forced already by an interval: balls of radii $r_1,\ldots,r_m$ cover an interval of length at most $2\sum_i r_i$. Norin and Turcotte~\cite[Section~6.1]{NorinTurcotte} also observed that the exact conclusion fails for an interval when $|T|<2r$. They conjectured that this is the only obstruction and stated the following as Conjecture~6.1.

\begin{conjecture}[Norin--Turcotte, Conjecture~6.1~\cite{NorinTurcotte}]\label{conj:norin-turcotte}
Let $r>0$ and let $T$ be a finite metric tree with $|T|\ge2r$. Then there is a probability measure $\nu$ on $\C(T,0)$ such that $E_\nu\le(|T|/r)U[0,r]$.
\end{conjecture}

Norin and Turcotte~\cite[Section~6.1]{NorinTurcotte} reported that they had verified \cref{conj:norin-turcotte} for metric trees with at most three leaves. They also reduced the conjecture to trees with no two-piece decomposition $T=T_1\cup T_2$ such that $|T_1|,|T_2|\ge2r$, and suggested extending their local two-ball construction beyond the minimal case. We address the remaining passage from at most three leaves to arbitrary finite metric trees.

Our main theorem resolves their conjecture.

\begin{theorem}\label{thm:main}
Let $r>0$ and let $T$ be a finite metric tree of length $L=|T|\ge2r$. Then there is a probability measure $\nu$ on $\C(T,0)$ such that $E_\nu\le (L/r)U[0,r]$.
\end{theorem}

We organize the proof in three steps. First, we prove a metric sum-plus-maximum covering lemma, in the spirit of Land and Lu~\cite[Lemma~1]{LandLu} and Norin and Turcotte~\cite[Lemma~3.2]{NorinTurcotte}, and use it to obtain a largest-first trimming rule. Thus we can impose the pointwise condition $\sum_i r_i\le L$ after controlling the expected radius measure.

Second, we separate rounding from the geometry of the tree. For a closed incidence relation between compact demand and object spaces, we show that finite ideality permits a fractional cover to be rounded to a random finite cover without increasing its intensity. We first round finitely many demands and then use tightness and weak convergence. For metric-tree balls, Tamir~\cite[Theorem~1]{Tamir} supplies balancedness, while Cornu{\'e}jols~\cite[Theorem~6.16]{Cornuejols} supplies the required covering ideality; see also Fulkerson, Hoffman, and Oppenheim~\cite{FulkersonHoffmanOppenheim}. Our compact formulation is tailored to coordinatewise intensity domination. Related infinite covering and duality frameworks appear in Aharoni and Holzman~\cite{AharoniHolzman} and Rademacher, Toriello, and Vielma~\cite{RademacherTorielloVielma}.

Third, we construct the fractional marked-ball cover. Norin and Turcotte~\cite[Lemma~3.3 and proof of Theorem~4.1]{NorinTurcotte} provide the local two-ball family and the underlying bootstrapping algebra. We recast that algebra as a zero-error replacement certificate and prove that every unresolved scale is at most one half of the optimal-split parameter. We then combine balanced splits with a three-piece reduction, modeled on the optimal-split argument in the proof of their Theorem~6.4, to obtain a macro-recursion. The recursion preserves coverage and the radius budget, while the unresolved subtrees decrease geometrically in length. Their total weight therefore tends to zero; see \cref{rem:finite-depth}.

Thus the cited inputs are the metric-tree decomposition framework and local construction of Norin and Turcotte, Tamir's balancedness theorem, and classical balanced-matrix ideality. The new part of the proof is the zero-error certificate packaging, the uniform local scale reduction, the three-piece macro-recursion, and the resulting exact fractional cover for arbitrary finite metric trees. The compact rounding theorem is a self-contained formulation of the limiting step needed here; we do not claim a new general theory of infinite covering.

We also prove a radius-budget duality specialized to metric-tree balls. General fractional covering duality for infinite hypergraphs has a substantial literature; see Aharoni and Holzman~\cite{AharoniHolzman} and Rademacher, Toriello, and Vielma~\cite{RademacherTorielloVielma}. In our setting, a finite positive measure $\beta$ on $[0,R]$ is feasible exactly when $\sigma(T)\le\int_{[0,R]}\max_{v\in T}\sigma(B_T(v,s))\,d\beta(s)$ for every finite positive demand measure $\sigma$, and it is enough to test finite atomic measures.

We use the exact framework to derive several consequences. If $D=\diamT(T)$, the uniform conclusion holds for $0<r\le L-D/2$, and this extends $r\le L/2$ whenever $T$ is not an interval. We determine the exact range for equal-arm metric stars, characterize the one-ball regime, and derive deterministic, obstruction, rigidity, and stability statements. The interval threshold itself is already implicit in Norin and Turcotte's interval obstruction and their verification for trees with at most three leaves~\cite[Section~6.1]{NorinTurcotte}; our contribution there is the measure and geometric rigidity statement in \cref{thm:interval-rigidity}.

Theorem~\ref{thm:main} is an exact result for metric trees. It does not by itself prove \cref{conj:burning}. In the transfer from a metric tree to a discrete tree, a center may lie inside an edge, and the discretization of Norin and Turcotte~\cite[Lemma~5.4]{NorinTurcotte} enlarges radii. Thus the continuous loss is removed, while the final discrete step remains a separate problem.

The paper has five sections. \Cref{sec:preliminaries} develops the metric-tree preliminaries, the strengthened trimming lemma, and compact ideal-cover rounding. \Cref{sec:local} proves the exact local replacement statements and completes the recursive proof of \cref{thm:main}. \Cref{sec:duality} establishes budgeted random--fractional duality and the concentration criterion. \Cref{sec:consequences} develops volume and packing obstructions, structural extensions, exact thresholds, deterministic consequences, rigidity, and stability.

\section{Preliminaries and compact rounding}\label{sec:preliminaries}
\subsection{Metric-tree notation and trimming}

\begin{definition}
A \emph{finite metric tree} is the metric realization of a finite combinatorial tree whose edges have positive lengths; models that differ only by edge subdivision represent the same metric tree. A \emph{finite combinatorial model} of $T$ is any such weighted combinatorial tree with metric realization $T$. We also allow the one-point tree and assign it length zero. The total edge length of $T$ is denoted by $|T|$.

For $x\in T$, the degree $\deg_T(x)$ is the number of components of $T\setminus\{x\}$. A point of degree one is a \emph{leaf}, a point of degree at least three is a \emph{branch point}, and $L(T)$ denotes the set of leaves. For $x,y\in T$, the unique arc joining them is denoted by $T[x,y]$ and has length $d_T(x,y)$. For $v\in T$ and $s\ge0$, the closed metric ball is $B_T(v,s)=\{x\in T:d_T(x,v)\le s\}$. If $A\subseteq S\subseteq T$, then $\cl_S(A)$ denotes the closure of $A$ in $S$.

For a nontrivial metric tree $S$, its diameter is $\diamT(S)=\max_{x,y\in S}d_S(x,y)$, the eccentricity of $v\in S$ is $\max_{x\in S}d_S(v,x)$, and its metric radius is $\rho(S)=\min_{v\in S}\max_{x\in S}d_S(v,x)$. A point attaining this minimum is a \emph{center} of $S$.
\end{definition}

\begin{definition}
A \emph{metric subtree} of $T$ is a nonempty closed connected subset of $T$. A finite family $\{T_1,\ldots,T_k\}$ of metric subtrees is a \emph{decomposition} of $T$ if $T=\bigcup_iT_i$ and $|T|=\sum_i|T_i|$. A \emph{two-piece decomposition} is a decomposition $T=T_1\cup T_2$ with $T_1\cap T_2$ consisting of one point. For a proper metric subtree $J\subsetneq T$, write $\overline J=\cl_T(T\setminus J)$.
\end{definition}

Every metric subtree is geodesically convex. Hence, if $S\subseteq T$ is a metric subtree, $v\in S$, and $s\ge0$, then $B_S(v,s)=S\cap B_T(v,s)$. We use the following decomposition facts from Norin and Turcotte~\cite[Section~2]{NorinTurcotte}.

\begin{lemma}[Norin--Turcotte, Section~2~\cite{NorinTurcotte}]\label{lem:components-at-point}
Let $S$ be a metric subtree and let $v\in S$. If $C$ is a component of $S\setminus\{v\}$, then $C\cup\{v\}$ is a metric subtree. The same is true for the union of $\{v\}$ with the closures of any set of components of $S\setminus\{v\}$.
\end{lemma}

\begin{lemma}[Norin--Turcotte, Section~2~\cite{NorinTurcotte}]\label{lem:branch-boundary}
If $J$ and $\overline J$ are nontrivial metric subtrees, then $J\cap\overline J$ consists of one point. Hence $\{J,\overline J\}$ is a two-piece decomposition.
\end{lemma}

\begin{definition}
A proper metric subtree $J\subsetneq T$ for which both $J$ and $\overline J$ are nontrivial is a \emph{branch} of $T$. By \cref{lem:branch-boundary}, the unique point of $J\cap\overline J$ is its \emph{anchor}.
\end{definition}

\begin{definition}
Let $m\ge1$. A radius tuple $\mathbf r=(r_1,\ldots,r_m)\in\R_{\ge0}^m$ is a \emph{cover} of $T$ if there are centers $v_1,\ldots,v_m\in T$ such that $T\subseteq\bigcup_{i=1}^mB_T(v_i,r_i)$. Let $\C_m(T)$ be the set of all such tuples, put $\C_0(T)=\varnothing$, and use the disjoint union $\C(T)=\bigsqcup_{m\ge0}\C_m(T)$. For $\ell\in\R$, a cover is $\ell$-good if $\sum_i r_i\le |T|+\ell$, and $\C(T,\ell)$ denotes the set of $\ell$-good covers.

If $\nu$ is a finite Borel measure on $\C(T)$, its \emph{expected radius measure} is the possibly infinite Borel measure $E_\nu$ on $\R_{\ge0}$ given by $E_\nu(A)=\int_{\C(T)}|\{i:r_i\in A\}|\,d\nu(\mathbf r)$ for every Borel set $A$. For positive Borel measures $\mu$ and $\nu$ on the same space, $\mu\le\nu$ means $\mu(A)\le\nu(A)$ for every Borel set $A$.

For $0\le a<b$, let $U[a,b]$ be the uniform probability measure on $[a,b]$. For a finite positive measure $\mu$ on $\R_{\ge0}$, set $\wgt(\mu)=2\int_0^\infty s\,d\mu(s)\in[0,\infty]$; in particular, $\wgt(U[a,b])=a+b$. For $r>0$, define $\Lam{T}{r}=(|T|/r)U[0,r]$.
\end{definition}

For every $m\ge1$, the set $\C_m(T)$ is closed in $\R_{\ge0}^m$. Indeed, if $\mathbf r^{(n)}\to\mathbf r$ and the covers $\mathbf r^{(n)}$ have centers $\mathbf v^{(n)}\in T^m$, then a convergent subsequence of the centers gives centers for $\mathbf r$. We use the disjoint-union Borel structure on $\C(T)$. Each layer $\C_m(T,\ell)$ is closed, so $\C(T,\ell)$ is Borel. Countable additivity of $E_\nu$ follows from monotone convergence.

\begin{lemma}[Norin--Turcotte, Lemma~3.1~\cite{NorinTurcotte}]\label{lem:tree-radius}
Every nontrivial finite metric tree satisfies $\rho(S)=\diamT(S)/2$.
\end{lemma}

The assertion is trivial for the one-point tree. By \cref{lem:tree-radius}, every nontrivial metric tree $S$ lies in one ball of radius $\diamT(S)/2$. Since $\diamT(S)\le |S|$, one ball of radius $a$ covers every tree of length at most $2a$.

\begin{definition}\label{def:minimal-tree}
Following Norin and Turcotte~\cite[Section~2]{NorinTurcotte}, a metric tree $S$ is \emph{$a$-minimal}, for $a>0$, if $|S|\ge a$ and $S$ has a two-piece decomposition into subtrees of length at most $a$.
\end{definition}

We use the following direct consequence of Norin and Turcotte's minimal-piece and one-ball lemmas.

\begin{corollary}[Norin--Turcotte, Lemmas~2.4 and~3.1~\cite{NorinTurcotte}]\label{cor:removable-piece}
Let $a>0$ and let $T$ be a metric tree with $|T|\ge a$. Then $T$ has a two-piece decomposition $T=T_0\cup T_1$ such that $a\le |T_1|\le2a$. In particular, one ball of radius $a$ covers $T_1$.
\end{corollary}

Norin and Turcotte~\cite[Lemma~3.2]{NorinTurcotte} proved the conclusion below under the stronger hypothesis $|T|\le\sum_i r_i$, while Land and Lu~\cite[Lemma~1]{LandLu} proved the corresponding sum-plus-maximum statement for discrete trees. We use the following metric analogue.

\begin{lemma}[Metric analogue of Land--Lu, Lemma~1~\cite{LandLu}]\label{lem:sum-plus-max}
Let $m\ge1$, let $r_1,\ldots,r_m\ge0$, and put $S=\sum_i r_i$ and $M=\max_i r_i$. If $|T|\le S+M$, then $(r_1,\ldots,r_m)$ is a cover of $T$.
\end{lemma}

\begin{proof}
We adapt the pruning induction of Land and Lu~\cite[proof of Lemma~1]{LandLu}, using the metric minimal-piece lemma of Norin and Turcotte~\cite[Lemma~2.4]{NorinTurcotte}. If $T$ is the one-point tree, the conclusion is immediate. Assume that $|T|>0$. The hypothesis then implies that at least one radius is positive. We delete the zero radii, renumber the remaining radii, and argue by induction on their number. If only one positive radius remains, the conclusion follows from $|T|\le2r_1$.

Assume that at least two positive radii remain, and let $a$ be a smallest one. If $|T|<a$, then the ball of radius $a$ covers $T$. Assume $|T|\ge a$. By \cref{cor:removable-piece}, we write $T=T_0\cup T_1$ with $a\le|T_1|\le2a$, and use the radius $a$ on $T_1$. The remaining radii have sum $S'=S-a$ and maximum $M$, and $|T_0|=|T|-|T_1|\le |T|-a\le S'+M$. They cover $T_0$ by induction. Together the balls cover $T$.
\end{proof}

We apply the preceding sum-plus-maximum criterion to extract a $0$-good subcover from any cover whose radii are bounded by the tree length.

\begin{lemma}\label{lem:trimming}
Let $T$ be a finite metric tree of length $L>0$, and let
$\mathbf s=(s_1,\ldots,s_m)\in\C(T)$ satisfy $0\le s_i\le L$ for every $i$.
After reordering the radii, some subtuple of $\mathbf s$ belongs to
$\C(T,0)$.
\end{lemma}

\begin{proof}
Assume $s_1\ge s_2\ge\cdots\ge s_m$. If $\sum_i s_i\le L$, then we keep all radii. Otherwise, let $k$ be the least index such that
\begin{equation}\label{eq:trim-threshold}
 2s_1+\sum_{i=2}^k s_i\ge L.
\end{equation}
The index exists because the left side for $k=m$ is
$s_1+\sum_i s_i>L$. By \cref{lem:sum-plus-max}, the tuple
$(s_1,\ldots,s_k)$ can be recentered to cover $T$.

If $k=1$, then $2s_1\ge L$, and the retained one-ball cover is $0$-good because $s_1\le L$. Let $k\ge2$. The choice of $k$ gives $2s_1+\sum_{i=2}^{k-1}s_i<L$. Since $s_k\le s_1$, we have $\sum_{i=1}^k s_i=s_1+\sum_{i=2}^{k-1}s_i+s_k<L-s_1+s_k\le L$. Thus the retained tuple is $0$-good.
\end{proof}

For fixed $m$, sorting the coordinates of $[0,L]^m$ is continuous. On the set where the total sum is greater than $L$, the first index satisfying \eqref{eq:trim-threshold} is Borel. Thus \cref{lem:trimming} gives a Borel trimming map on every layer.

\begin{corollary}\label{cor:random-trimming}
Let $T$ be a finite metric tree of length $L>0$, let $0<r\le L$, and suppose that $\nu_0$ is a probability measure on $\C(T)$ such that almost every sampled radius lies in $[0,r]$. Then there is a probability measure $\nu$ on $\C(T,0)$ such that $E_\nu\le E_{\nu_0}$.
\end{corollary}

\begin{proof}
We apply the Borel trimming map on the full-measure set of covers whose radii lie in $[0,r]\subseteq[0,L]$ and map the null complement to the one-ball cover $(L/2)$. The resulting push-forward is a probability measure on $\C(T,0)$. The map only deletes radii almost surely, so $E_\nu(A)\le E_{\nu_0}(A)$ for every Borel set $A$.
\end{proof}

\begin{remark}
The trimming step keeps radii, not centers. The retained radii may need new centers. This is why the bound for each cover can be recovered after the marked centers are forgotten.
\end{remark}

\subsection{Compact ideal-cover rounding}

We first state the rounding argument without reference to metric trees. Let $Y$ and $X$ be nonempty compact metric spaces, and let $R\subseteq Y\times X$ be closed. A point of $Y$ is a demand and a point of $X$ is a covering object. Put $R_y=\{x\in X:(y,x)\in R\}$ and $R^x=\{y\in Y:(y,x)\in R\}$. A finite counting measure $\eta$ on $X$ covers $Y$ if $\eta(R_y)\ge1$ for every $y\in Y$. A finite positive Borel measure $\xi$ on $X$ is a fractional cover if $\xi(R_y)\ge1$ for every $y\in Y$.

Following Tamir~\cite{Tamir}, a $0$--$1$ matrix is \emph{balanced} if it has no square submatrix of odd order in which every row and every column has exactly two ones. For finite families $y_1,\ldots,y_n\in Y$ and $x_1,\ldots,x_k\in X$, let $A$ be the incidence matrix $A_{ij}=\one_{\{(y_i,x_j)\in R\}}$. We call $R$ \emph{ideal} if, for every such matrix $A$, every extreme point of the set-covering polyhedron $P(A)=\{z\in\R_{\ge0}^k:Az\ge\mathbf1\}$ is integral.

For a compact metric space $X$, let $\Nf(X)$ be the space of finite counting measures $\eta=\sum_{i=1}^m\delta_{x_i}$, where $m\ge0$, with the weak topology inherited from the finite positive Borel measures on $X$. For every Borel set $A\subseteq X$, the evaluation map $\eta\mapsto\eta(A)$ is Borel. Indeed, it is upper semicontinuous for closed $A$ by the portmanteau theorem, and the class of sets with Borel evaluation maps is a Dynkin system containing the closed sets, which form a generating $\pi$-system. See Reiss~\cite[Section~1.1]{Reiss} for the standard point-process framework and Billingsley~\cite[Theorem~2.1]{Billingsley} for the portmanteau theorem. For a probability measure $\mathsf P$ on $\Nf(X)$ with finite mean mass, its \emph{intensity} is $I_{\mathsf P}(A)=\int_{\Nf(X)}\eta(A)\,d\mathsf P(\eta)$. By Tonelli's theorem, we have
\begin{equation}\label{eq:campbell-finite}
 \int_X f\,dI_{\mathsf P}
 =\int_{\Nf(X)}\left(\int_X f\,d\eta\right)d\mathsf P(\eta)
\end{equation}
for every nonnegative Borel function $f$. We also use the following standard form of measure order on a compact space:
\begin{equation}\label{eq:measure-order-standard}
 \int f\,d\mu\le\int f\,d\nu
 \quad(f\in C(X),\ f\ge0)
 \quad\Longrightarrow\quad \mu\le\nu.
\end{equation}
This follows from the Riesz--Markov theorem.

\begin{lemma}[Standard point-measure compactness; see Reiss and Billingsley~\cite{Reiss,Billingsley}]\label{lem:point-measure-compactness}
For every integer $K\ge0$, the set $\Nf^{\le K}(X)=\{\eta\in\Nf(X):\eta(X)\le K\}$ is compact. The space $\Nf(X)$ is a closed Polish subspace of the finite positive Borel measures on $X$. Moreover, if $F\subseteq X$ is closed, then $H_F=\{\eta\in\Nf(X):\eta(F)\ge1\}$ is closed in $\Nf(X)$.
\end{lemma}

We first use the finite-dimensional rounding step for an ideal set-covering polyhedron; compare Cornu{\'e}jols~\cite[Chapter~6]{Cornuejols}.

\begin{lemma}\label{lem:finite-demand-rounding}
Let $D\subseteq Y$ be finite and nonempty, and let $\xi$ be a finite fractional cover on $X$. If the incidence relation is ideal, then there is a probability measure $\mathsf P_D$ on $\Nf(X)$ such that every realization covers $D$ and $I_{\mathsf P_D}\le\xi$.
\end{lemma}

\begin{proof}
We write $D=\{y_1,\ldots,y_n\}$. For every nonempty $I\subseteq[n]$, let $Q_I=\{x\in X:\{i:(y_i,x)\in R\}=I\}$. Each $Q_I$ is Borel. We ignore $Q_\varnothing$ and the classes of $\xi$-measure zero, enumerate the remaining classes as $Q_1,\ldots,Q_k$, and set $z_j=\xi(Q_j)>0$. We choose $x_j\in Q_j$, and let $A$ be the resulting incidence matrix. Since $\xi$ fractionally covers $D$, we have $Az\ge\mathbf1$, so $z\in P(A)$.

By ideality, every vertex of $P(A)$ is integral. The polyhedron is pointed and its recession cone is $\R_{\ge0}^k$. The Minkowski--Weyl theorem therefore yields vertices $y^{(1)},\ldots,y^{(N)}\in\mathbb Z_{\ge0}^k$ of $P(A)$, numbers $p_\ell\ge0$ with $\sum_\ell p_\ell=1$, and $h\in\R_{\ge0}^k$ such that $z=\sum_{\ell=1}^Np_\ell y^{(\ell)}+h$. Set $\bar y_j^{(\ell)}=\min\{1,y_j^{(\ell)}\}$. Since $A$ is a $0$--$1$ matrix and $Ay^{(\ell)}\ge\mathbf1$, every row contains a column on which $\bar y^{(\ell)}$ equals one. Hence $A\bar y^{(\ell)}\ge\mathbf1$ and $\sum_{\ell=1}^Np_\ell\bar y^{(\ell)}\le\sum_{\ell=1}^Np_\ell y^{(\ell)}\le z$.

Let $J_\ell=\{j:\bar y_j^{(\ell)}=1\}$. Conditional on $\ell$, we independently choose one point of each $Q_j$, $j\in J_\ell$, according to the probability measure $\xi|_{Q_j}/z_j$, and take the corresponding counting measure. Formally, on $\Omega_\ell=\prod_{j\in J_\ell}Q_j$, we use the product measure $\kappa_\ell=\bigotimes_{j\in J_\ell}(\xi|_{Q_j}/z_j)$ and the Borel map $\Psi_\ell((x_j)_{j\in J_\ell})=\sum_{j\in J_\ell}\delta_{x_j}$.
Define $\mathsf P_D=\sum_\ell p_\ell(\Psi_\ell)_*\kappa_\ell$. Every realization covers $D$. On $Q_j$ its intensity is
\[
 I_{\mathsf P_D}|_{Q_j}
 =\left(\sum_{\ell=1}^Np_\ell\bar y_j^{(\ell)}\right)
   \frac{\xi|_{Q_j}}{z_j}
 \le\xi|_{Q_j}.
\]
It assigns no mass to $Q_\varnothing$, and hence $I_{\mathsf P_D}\le\xi$ on all of $X$.
\end{proof}

Aharoni and Holzman~\cite{AharoniHolzman} and Rademacher, Toriello, and Vielma~\cite{RademacherTorielloVielma} give broader background on fractional covering and infinite-dimensional covering polyhedra. We now prove the compact intensity-domination form used below.

\begin{theorem}\label{thm:compact-rounding}
Let $Y$ and $X$ be nonempty compact metric spaces, and let $R\subseteq Y\times X$ be a closed ideal incidence relation. If $\xi$ is a finite fractional cover on $X$, then there is a probability measure $\mathsf P$ on $\Nf(X)$ such that almost every configuration covers $Y$ and $I_{\mathsf P}\le\xi$. In particular, the expected number of selected objects is at most $\xi(X)$.
\end{theorem}

\begin{proof}
Choose a dense sequence $q_1,q_2,\ldots$ in $Y$, put $D_n=\{q_1,\ldots,q_n\}$, and let $\mathsf P_n$ be given by \cref{lem:finite-demand-rounding}. Write $M=\xi(X)$. Since $I_{\mathsf P_n}\le\xi$, we have $\E_{\mathsf P_n}\eta(X)\le M$. For every positive integer $K$, by Markov's inequality and \cref{lem:point-measure-compactness}, we have $\mathsf P_n(\Nf(X)\setminus\Nf^{\le K}(X))\le M/K$. Thus $(\mathsf P_n)$ is tight. Prokhorov's theorem, in the form stated by Billingsley~\cite[Theorem~5.1]{Billingsley}, gives a subsequence $\mathsf P_{n_j}$ converging weakly to a probability measure $\mathsf P$ on $\Nf(X)$.

Fix $q=q_i$. The section $R_q$ is closed, and therefore
$H_q=\{\eta:\eta(R_q)\ge1\}$ is closed by \cref{lem:point-measure-compactness}. For all sufficiently large $j$, $\mathsf P_{n_j}(H_q)=1$. The portmanteau theorem gives $\mathsf P(H_q)=1$. After intersecting these full-measure events over all $i$, a limiting configuration $\eta=\sum_{h=1}^m\delta_{x_h}$ covers every $q_i$. Its covered demand set is $\bigcup_{h=1}^mR^{x_h}$, a finite union of closed sets. It contains a dense subset of $Y$, and hence it equals $Y$. Thus $\mathsf P$ is supported on finite covers of $Y$.

Let $f\in C(X)$ be nonnegative and put $G_f(\eta)=\int f\,d\eta$. The function $G_f$ is continuous and nonnegative, although it need not be bounded on $\Nf(X)$. For $K>0$, $G_f\wedge K$ is bounded and continuous. Hence weak convergence and monotone convergence give
\[
 \int G_f\,d\mathsf P
 =\sup_K\lim_{j\to\infty}\int(G_f\wedge K)\,d\mathsf P_{n_j}
 \le\liminf_{j\to\infty}\int G_f\,d\mathsf P_{n_j}
 \le\int f\,d\xi.
\]
Taking $f=1$ shows that $\mathsf P$ has finite mean mass. By \eqref{eq:campbell-finite}, the left side is $\int f\,dI_{\mathsf P}$; \eqref{eq:measure-order-standard} now gives $I_{\mathsf P}\le\xi$.
\end{proof}

We now specialize the theorem to metric-tree balls.

\begin{definition}
Fix a finite metric tree $T$ and $r>0$, and put $X_{T,r}=T\times[0,r]$. A point $b=(v,s)\in X_{T,r}$ is a \emph{marked ball}, with center $v$, radius $s$, and underlying ball $B(b)=B_T(v,s)$. The radius map is $\pi:X_{T,r}\to[0,r]$, $\pi(v,s)=s$. For a finite positive Borel measure $\xi$ on $X_{T,r}$, its coverage function is $c_\xi(x)=\int_{X_{T,r}}\one_{\{x\in B(b)\}}\,d\xi(b)$.

The measure $\xi$ is a \emph{fractional marked-ball cover} of $T$ if $c_\xi(x)\ge1$ for every $x\in T$; its \emph{radius measure} is $\pi_*\xi$. A finite counting measure $\eta\in\Nf(X_{T,r})$ is a \emph{finite marked cover} if $c_\eta(x)\ge1$ for every $x\in T$. A \emph{random finite marked cover} is a probability measure $\mathsf P$ on $\Nf(X_{T,r})$, with finite mean mass, that is supported on finite marked covers. Its expected radius measure is $\pi_*I_{\mathsf P}$.
\end{definition}

Fix $m\ge1$ and put $K_m=\C_m(T)\cap[0,r]^m$. The relation
\[
 \mathcal F_m=\left\{(\mathbf s,\mathbf v)\in K_m\times T^m:
 T\subseteq\bigcup_{i=1}^mB_T(v_i,s_i)\right\}
\]
is closed. Indeed, the function
\[
 (\mathbf s,\mathbf v)\longmapsto
 \max_{x\in T}\min_{1\le i\le m}\bigl(d_T(x,v_i)-s_i\bigr)
\]
is continuous, and $\mathcal F_m$ is its nonpositive sublevel set. Every section over $K_m$ is nonempty and compact. The Arsenin--Kunugui selection theorem, in the form stated by Kechris~\cite[Theorem~18.18]{Kechris}, gives a Borel map
\begin{equation}\label{eq:borel-lifting}
 \mathbf v_m:K_m\longrightarrow T^m
\end{equation}
whose selected centers realize the cover.

For the marked-ball space, write $R_{T,r}=\{(x,(v,s))\in T\times X_{T,r}:d_T(x,v)\le s\}$. In Tamir's tree-network terminology, closed metric balls, including radius-zero balls, are neighborhood subtrees~\cite{Tamir}. Together with balanced-matrix ideality~\cite[Theorem~6.16]{Cornuejols}, this gives the finite incidence property below.

\begin{lemma}\label{lem:tree-ball-ideality}
For finite families $x_1,\ldots,x_n\in T$ and $(v_1,s_1),\ldots,(v_k,s_k)\in X_{T,r}$, the matrix $A_{ij}=\one_{\{d_T(x_i,v_j)\le s_j\}}$ is balanced. Consequently, every extreme point of $P(A)=\{z\in\R_{\ge0}^k:Az\ge\mathbf1\}$ is integral. Hence the incidence relation $(x,(v,s))\in R_{T,r}$ if and only if $d_T(x,v)\le s$ is closed and ideal.
\end{lemma}

\begin{proof}
Choose a finite combinatorial model of $T$ and subdivide it at all demand points $x_i$, all centers $v_j$, and all points of the spheres $\{x\in T:d_T(x,v_j)=s_j\}$. There are only finitely many such subdivision points: on each edge, the distance from a fixed center is piecewise affine with slopes in $\{-1,1\}$, so a fixed level is met at most twice. Subdivision leaves the metric and all incidences unchanged.

Each singleton $\{x_i\}$ is a radius-zero neighborhood subtree, and each $B_T(v_j,s_j)$ is a neighborhood subtree of the subdivided tree. Moreover, $A_{ij}=1$ if and only if $\{x_i\}\cap B_T(v_j,s_j)\ne\varnothing$. Thus $A$ is the intersection matrix of two finite families of neighborhood subtrees. Tamir~\cite[Theorem~1]{Tamir} proved that every such matrix is balanced.

By Cornu{\'e}jols~\cite[Theorem~6.16]{Cornuejols}, every submatrix of a balanced $0$--$1$ matrix is ideal; see also Fulkerson, Hoffman, and Oppenheim~\cite{FulkersonHoffmanOppenheim}. In particular, the polytope $Q(A)=\{z\in[0,1]^k:Az\ge\mathbf1\}$ is integral. Let $z$ be an extreme point of $P(A)$. If $z_j>1$ for some $j$, then every covering inequality involving column $j$ has slack at least $z_j-1$. Hence, for all sufficiently small $\varepsilon>0$, both $z+\varepsilon e_j$ and $z-\varepsilon e_j$ belong to $P(A)$, contrary to the extremality of $z$. Thus $z\in[0,1]^k$. Since an extreme point of $P(A)$ lying in $Q(A)$ is also an extreme point of $Q(A)$, it is integral. Therefore the incidence relation is ideal. Finally, $R_{T,r}$ is the inverse image of $(-\infty,0]$ under the continuous map $(x,v,s)\mapsto d_T(x,v)-s$, and is consequently closed.
\end{proof}

For $m\ge0$, let $\Nf^m([0,r])=\{\theta\in\Nf([0,r]):\theta([0,r])=m\}$. If $m\ge1$, $\theta\in\Nf^m([0,r])$, and $1\le j\le m$, define its $j$th largest atom by
\[
 s_j(\theta)=\sup\bigl(\{t\in\mathbb Q\cap[0,r]:
 \theta([t,r])\ge j\}\cup\{0\}\bigr).
\]
The evaluation-map facts above show that every $s_j$ is Borel. With the empty tuple on $\Nf^0([0,r])$, the layerwise maps $\operatorname{rad}_m(\theta)=(s_1(\theta),\ldots,s_m(\theta))$ define a Borel radius-list map $\operatorname{rad}:\Nf([0,r])\to\bigsqcup_{m\ge0}[0,r]^m$.

\begin{corollary}\label{cor:balanced-rounding}
Let $\xi$ be a fractional marked-ball cover of $T$ on $X_{T,r}$. Then there is a random finite marked cover $\mathsf P$ such that $I_{\mathsf P}\le\xi$. After the centers are forgotten, there is a probability measure $\nu_0$ on $\C(T)$ such that every sampled radius lies in $[0,r]$ and $E_{\nu_0}\le\pi_*\xi$.
\end{corollary}

\begin{proof}
By \cref{lem:tree-ball-ideality}, the relation $R_{T,r}$ satisfies the hypotheses of \cref{thm:compact-rounding}; we apply that theorem. We then forget the centers in a Borel way. The map $\eta\mapsto\pi_*\eta$ is continuous, so $R(\eta)=\operatorname{rad}(\pi_*\eta)$ is Borel.

Choose a finite $r/2$-net in $T$ and place a ball of radius $r$ at every point of the net. This gives a fixed finite cover $\mathbf c^*$ whose radii lie in $[0,r]$. Define
\[
 R_T(\eta)=
 \begin{cases}
 R(\eta),&R(\eta)\in\C(T),\\
 \mathbf c^*,&R(\eta)\notin\C(T).
 \end{cases}
\]
The map is Borel because each $\C_m(T)$ is closed. The second line is used only on a null set. Let $\nu_0=(R_T)_*\mathsf P$. For every Borel set $A\subseteq[0,r]$, By Tonelli's theorem, we have
\[
 E_{\nu_0}(A)
 =\int(\pi_*\eta)(A)\,d\mathsf P(\eta)
 =I_{\mathsf P}(T\times A)
 =(\pi_*I_{\mathsf P})(A).
\]
Therefore $E_{\nu_0}=\pi_*I_{\mathsf P}\le\pi_*\xi$.
\end{proof}

If the target radius does not exceed the total length of the tree, the trimming lemma allows us to remove the allowance and obtain a $0$-good random cover.

\begin{corollary}\label{cor:budgeted-rounding}
Let $T$ be a finite metric tree of length $L>0$, let $0<r\le L$, and let $\xi$ be a fractional marked-ball cover on $X_{T,r}$. Then there is a probability measure $\nu$ on $\C(T,0)$ such that $E_\nu\le\pi_*\xi$.
\end{corollary}

\begin{proof}
Apply \cref{cor:balanced-rounding}, and then apply \cref{cor:random-trimming}.
\end{proof}

\section{Local replacement and the proof of the main theorem}\label{sec:local}
\subsection{Exact replacement and tree decompositions}

If $0<s<r$, then we view a measure on $T\times[0,s]$ as a measure on $T\times[0,r]$ by the natural inclusion.

\begin{definition}
A replacement certificate for $(T,r)$ consists of a finite positive measure $\eta$ on $X_{T,r}$, scales $0<s_1,\ldots,s_q<r$, and weights $w_1,\ldots,w_q\ge0$ such that $c_\eta(x)+\sum_{j=1}^q w_j\ge1$ for every $x\in T$, and
\begin{equation}\label{eq:replacement-budget}
 \pi_*\eta+\sum_{j=1}^q w_j\Lam{T}{s_j}\le\Lam{T}{r}.
\end{equation}
Terms with zero weight are omitted.
\end{definition}

We now recast the bootstrapping calculation in the proof of Norin and Turcotte's Theorem~4.1 as a zero-error certificate that records the unresolved scales.

\begin{lemma}[After Norin--Turcotte, proof of Theorem~4.1~\cite{NorinTurcotte}]\label{lem:exact-replacement}
Let $|T|=L>0$ and $r>0$. Suppose that a random finite marked cover $\mathsf P_0$ of $T$ has expected radius measure
\[
 \pi_*I_{\mathsf P_0}=\sum_{i=1}^k\alpha_i U[a_i,r],
 \qquad 0\le a_i<r,
 \quad \alpha_i\ge0,
\]
and $\sum_{i=1}^k\alpha_i(r+a_i)\le L$. Then $(T,r)$ has a replacement certificate whose positive scales are among $a_1,\ldots,a_k$.
\end{lemma}

\begin{proof}
We use the mixing weights from the proof of Norin and Turcotte's Theorem~4.1~\cite{NorinTurcotte}, but retain the unresolved scales instead of inserting an error factor. Let $I_+=\{i:a_i>0\}$. Set $D=L+\sum_i\alpha_i a_i^2/(r-a_i)$, $p_0=L/D$, and $p_i=\alpha_i a_i^2/((r-a_i)D)$. Then $p_0+\sum_i p_i=1$, and $p_i=0$ for $i\notin I_+$. Put $\eta=p_0I_{\mathsf P_0}$; for $i\in I_+$, set $w_i=p_i$ and $s_i=a_i$. Thus every retained scale is positive. Since every realization of $\mathsf P_0$ covers $T$, we have $c_{I_{\mathsf P_0}}\ge1$, and hence $c_\eta+\sum_{i\in I_+}w_i\ge p_0+\sum_{i\in I_+}p_i=1$.

For $i\in I_+$, we have $p_iL/a_i=p_0\alpha_i a_i/(r-a_i)$. Separating the indices with $a_i=0$ gives
\begin{align*}
 \pi_*\eta+\sum_{i\in I_+}w_i\Lam{T}{a_i}
 &=p_0\sum_i\alpha_i U[a_i,r]
   +\sum_{i\in I_+}p_i\frac{L}{a_i}U[0,a_i]\\
 &=p_0\sum_{i\in I_+}\alpha_i
   \left(U[a_i,r]+\frac{a_i}{r-a_i}U[0,a_i]\right)
   +p_0\sum_{i\notin I_+}\alpha_iU[0,r]\\
 &=p_0\left(\sum_i\frac{\alpha_i r}{r-a_i}\right)U[0,r].
\end{align*}
The last identity also includes $a_i=0$, for which $\alpha_i r/(r-a_i)=\alpha_i$. We only need to show that the final coefficient is at most $L/r$. After multiplication by $rD/L$, the required inequality is $r^2\sum_i\alpha_i/(r-a_i)\le L+\sum_i\alpha_i a_i^2/(r-a_i)$. Indeed, after moving the second sum to the left, its left-hand side becomes $\sum_i\alpha_i(r^2-a_i^2)/(r-a_i)=\sum_i\alpha_i(r+a_i)\le L$. Thus \eqref{eq:replacement-budget} holds.
\end{proof}

The concatenation and mixture rules of Norin and Turcotte give the following certificate reformulation.

\begin{lemma}[Norin--Turcotte, Lemma~4.2(b),(c)~\cite{NorinTurcotte}]\label{lem:composition}
Let a certificate for $(T,r)$ have resolved measure $\eta$ and remaining terms $(w_j,s_j)$. For each $j$, let a certificate for $(T,s_j)$ have resolved measure $\eta_j$ and remaining terms $(u_{jh},t_{jh})$. Then $\eta'=\eta+\sum_j w_j\eta_j$ and the terms $(w_j u_{jh},t_{jh})$ form a certificate for $(T,r)$.
\end{lemma}

For a nontrivial metric tree $T$, define
\begin{equation}\label{eq:lambda-def}
 \lambda(T)=\inf_{T=T_1\cup T_2}\max\{|T_1|,|T_2|\},
\end{equation}
where the infimum is over all two-piece decompositions. A two-piece decomposition attaining this infimum is an \emph{optimal split}.

The proof of Norin and Turcotte's Theorem~6.4 contains the optimal-split facts that we need. We record them in our notation.

\begin{lemma}[Norin--Turcotte, proof of Theorem~6.4~\cite{NorinTurcotte}]\label{lem:lambda-basic}
If $T$ is a nontrivial finite metric tree of length $L$, then the infimum in \eqref{eq:lambda-def} is attained and $L/2\le\lambda(T)\le2L/3<L$. Consequently, every optimal split has two nontrivial sides, and $T$ is $\lambda(T)$-minimal in the terminology of Norin and Turcotte~\cite{NorinTurcotte}.
\end{lemma}

The lower bound follows because the two side lengths sum to $L$; the cited proof gives attainment, the upper bound $2L/3$, and the resulting minimality.

\begin{lemma}\label{lem:interval-split}
A nontrivial finite metric tree has two leaves if and only if it is a metric interval. If an interval $I$ has length $L$, then $\lambda(I)=L/2$, and every optimal split has two sides of length $L/2$.
\end{lemma}

\begin{proof}
An interval has exactly its two endpoints as leaves. Conversely, suppress the degree-two vertices in a finite combinatorial model of a metric tree. For the resulting finite tree, $|L(T)|=2+\sum_{v:\,\deg(v)\ge3}(\deg(v)-2)$. Thus a tree with exactly two leaves has no branch vertex and is a path, hence a metric interval.

Every two-piece decomposition of an interval is obtained by cutting at one point. If the two side lengths are $x$ and $L-x$, then
$\max\{x,L-x\}\ge L/2$, while the midpoint realizes equality. Hence $\lambda(I)=L/2$. In an optimal split the two side lengths sum to $L$ and both are at most $L/2$, so both equal $L/2$.
\end{proof}

We use the following local lemma of Norin and Turcotte.

\begin{lemma}[Norin--Turcotte, Lemma~3.3~\cite{NorinTurcotte}]\label{lem:NTlocal}
Let $|T|=L$ and $\lambda=\lambda(T)$. There is a number $0\le a\le\min\{\lambda/2-L/4,L/12\}$ such that, with $A=L/4+a$, $B=L/4-3a$, and $S=A+B=L/2-2a$, the pair $(B-x,A+x)$ covers $T$ for every $0\le x\le B$, and one ball of every radius $y\ge S$ covers $T$.
\end{lemma}

We combine Norin and Turcotte's local family and mixing rules~\cite[Lemmas~3.3 and~4.2, and proof of Theorem~4.1]{NorinTurcotte} with \cref{lem:exact-replacement}. The new conclusion is the uniform recursive bound on every unresolved scale.

\begin{proposition}\label{prop:local-replacement}
Let $|T|=L\ge2r$ and suppose that $\lambda(T)<2r$. Then $(T,r)$ has a replacement certificate in which every positive scale $s_j$ satisfies $s_j\le\lambda(T)/2$.
\end{proposition}

\begin{proof}
We use $a,A,B,S$ from Norin and Turcotte's local lemma, stated as \cref{lem:NTlocal}. Since $A=L/4+a\le\lambda(T)/2<r$ and $r\le L/2$, we have two cases.

Suppose first that $A<r\le S$, and put $C=S-r$. Then $0\le C<B\le A$. For $0\le t\le1$, set $x=t(r-A)$. Since $r-A\le B$, the pair $(B-t(r-A),A+t(r-A))$ covers $T$. Increase the first radius to $A-t(r-A+4a)$. The increase is $4a(1-t)\ge0$. Thus the new pair is still a cover. Its first coordinate runs from $A$ to $C$, and its second coordinate runs from $A$ to $r$. Uniform $t$ and the Borel lifting in \eqref{eq:borel-lifting} give a random finite marked cover $\mathsf P_0$ with $\pi_*I_{\mathsf P_0}=U[C,A]+U[A,r]$. We have
\begin{equation}\label{eq:local-decomp}
 U[C,A]+U[A,r]=\alpha_CU[C,r]+\alpha_AU[A,r],
\end{equation}
where $\alpha_C=(r-C)/(r-B)$ and $\alpha_A=4a/(r-B)$. Applying $\wgt$ to \eqref{eq:local-decomp} gives $\alpha_C(r+C)+\alpha_A(r+A)=(C+A)+(A+r)=3A+B=L$. Thus \cref{lem:exact-replacement} applies. The possible remaining scales are $C$ and $A$; every positive remaining scale is at most $A\le\lambda(T)/2$.

Now suppose that $S<r$. Every $y\in[S,r]$ gives a one-ball cover. Using the layerwise Borel center selection in \eqref{eq:borel-lifting}, a uniform choice of $y$ gives a random finite marked cover with expected radius measure $U[S,r]$. Since $S+r\le L$, \cref{lem:exact-replacement} gives a certificate whose only possible positive scale is $S$. If $S\le\lambda(T)/2$, we are done.

Assume $S>\lambda(T)/2$. We first show that $B>0$. Otherwise $a=L/12$. Since $a\le\lambda(T)/2-L/4$ and $\lambda(T)\le2L/3$, this forces $\lambda(T)=2L/3$. It then gives $S=L/2-2a=L/3=\lambda(T)/2$, a contradiction. Thus $B>0$, and consequently $A<S=A+B$. We apply the first case with target radius $S$. In that application the new value of $C$ is zero, so its only possible positive remaining scale is $A\le\lambda(T)/2$. Compose the two certificates by \cref{lem:composition}.
\end{proof}

We next describe the two possible forms of an optimal split. By \cref{lem:lambda-basic}, both sides are nontrivial. Fix an optimal split
\begin{equation}\label{eq:optimal-split}
 T=A\cup B,
 \qquad A\cap B=\{v\},
 \qquad |A|=\lambda\ge b=|B|>0.
\end{equation}
Then $|T|=\lambda+b$ and $\lambda\ge|T|/2$. The split is balanced if $\lambda=b=|T|/2$, and non-balanced if $\lambda>b$.

The optimal-split argument of Norin and Turcotte~\cite[proof of Theorem~6.4]{NorinTurcotte} gives the following elementary gap property.

\begin{lemma}\label{lem:branch-gap}
No branch of $T$ has length strictly between $b$ and $\lambda$.
\end{lemma}

\begin{proof}
Using the optimality criterion from Norin and Turcotte~\cite[proof of Theorem~6.4]{NorinTurcotte}, suppose that a branch $J$ satisfies $b<|J|<\lambda$. Then $\{J,\overline J\}$ is a two-piece decomposition and $b<|\overline J|=|T|-|J|=b+\lambda-|J|<\lambda$. Both sides have length less than $\lambda$, contrary to the choice of \eqref{eq:optimal-split}.
\end{proof}

We record a consequence of optimality.

\begin{lemma}\label{lem:leaf-monotonicity}
If $S=S_1\cup S_2$ is a two-piece decomposition into nontrivial subtrees, then $|L(S_i)|\le|L(S)|$ for $i=1,2$.
\end{lemma}

\begin{proof}
Let $x=S_1\cap S_2$. Every leaf of $S_1$ other than possibly $x$ is a leaf of $S$, while the omitted subtree $S_2$ contains a leaf of $S$. Hence a possible new leaf at $x$ replaces an omitted leaf. The argument for $S_2$ is the same.
\end{proof}

We next refine the branch-transfer argument in the proof of Norin and Turcotte's Theorem~6.4~\cite{NorinTurcotte}. Its pairwise-union and leaf-loss conclusions are the form needed for our recursion.

\begin{proposition}\label{prop:three-piece}
Assume that $\lambda>b$ in \eqref{eq:optimal-split}. Then there are three nontrivial subtrees $T_1,T_2,T_3$, all containing $v$, such that
\[
 T=T_1\cup T_2\cup T_3,
 \qquad T_i\cap T_j=\{v\}\quad(i\ne j),
 \qquad |T_i|\le b.
\]
Each pairwise union has length at least $\lambda$ and has fewer leaves than $T$.
\end{proposition}

\begin{proof}
We refine the branch-transfer argument in the proof of Norin and Turcotte's Theorem~6.4~\cite{NorinTurcotte}. The set $A\setminus\{v\}$ has at least two components. Otherwise it has one component. Choose a sufficiently small $0<\varepsilon<\lambda-b$ and move the cut point a distance $\varepsilon$ into that component. The new side lengths are $\lambda-\varepsilon$ and $b+\varepsilon$, both less than $\lambda$, a contradiction.

Choose one component $C$ of $A\setminus\{v\}$, and put $D=A\setminus(C\setminus\{v\})$. By \cref{lem:components-at-point}, both $C\cup\{v\}$ and $D$ are subtrees. Set $T_1=B$, $T_2=C\cup\{v\}$, and $T_3=D$. All three are nontrivial. The complements of $T_2$ and $T_3$ are unions of the other components through $v$, so they are connected. By \cref{lem:branch-boundary}, $T_2$ and $T_3$ are branches of $T$. Both have length less than $\lambda$. By \cref{lem:branch-gap}, their lengths are at most $b$. Also $|T_1|=b$.

We have $|T_2\cup T_3|=\lambda$ and $|T_1\cup T_2|=|T|-|T_3|\ge|T|-b=\lambda$, with the same bound for $T_1\cup T_3$.

A pairwise union omits one nontrivial branch. Such a branch contains a leaf of $T$. Indeed, choose a point $z$ in the branch maximizing its distance from the anchor. Then $z$ is not the anchor; if $z$ were not a leaf of $T$, an incident edge lying in the same component would continue beyond $z$, contradicting maximality. Every leaf of the union other than possibly $v$ is a leaf of $T$, and $v$ is not a leaf because two nontrivial pieces remain incident with it. Thus at least one leaf is lost and no new leaf is added.
\end{proof}

If $S\subseteq T$ is a subtree, we view a measure on $S\times[0,r]$ as a measure on $T\times[0,r]$ by the inclusion $(v,s)\mapsto(v,s)$. Since $B_S(v,s)=S\cap B_T(v,s)$, this does not decrease coverage and does not change the radius measure.

\begin{proposition}\label{prop:half-sum}
Let $T=T_1\cup T_2\cup T_3$, where $T_i\cap T_j=\{v\}$ for $i\ne j$. Fix $r>0$. Suppose that every pairwise union $T_i\cup T_j$ has a fractional marked-ball cover $\xi_{ij}$ satisfying $\pi_*\xi_{ij}\le\Lam{T_i\cup T_j}{r}$. Then $\xi=\frac12(\xi_{12}+\xi_{13}+\xi_{23})$ is a fractional marked-ball cover of $T$ and $\pi_*\xi\le\Lam{T}{r}$.
\end{proposition}

\begin{proof}
A point in $T_i$ belongs to the two pairwise unions containing $T_i$, so its coverage is at least $1/2+1/2=1$. Also,
\begin{align*}
 \pi_*\xi
 &\le\frac{1}{2r}
 \bigl(|T_1\cup T_2|+|T_1\cup T_3|+|T_2\cup T_3|\bigr)U[0,r]\\
 &=\frac{|T|}{r}U[0,r]=\Lam{T}{r},
\end{align*}
where each $T_i$ is counted twice.
\end{proof}

If the optimal split in \eqref{eq:optimal-split} is balanced, then for every $r>0$,
\begin{equation}\label{eq:balanced-budget}
 \Lam{A}{r}+\Lam{B}{r}=\Lam{T}{r}.
\end{equation}

\subsection{The recursive construction}

For $m\ge2$, let $\mathsf P(m)$ denote the assertion that every finite metric tree with exactly $m$ leaves and length at least twice the target scale has a fractional marked-ball cover with the required uniform radius budget. We prove $\mathsf P(m)$ by strong induction on $m$. We use the following reduction at the $m$th induction step: it assumes only $\mathsf P(j)$ for $2\le j<m$, and it either resolves the current problem or replaces it by smaller subproblems with the same number of leaves while preserving coverage and budget.

There are two independent decreases. A non-balanced split is resolved using trees with fewer leaves. A balanced split may retain a tree with the same number of leaves, but its length is halved. In the local-replacement case, every new scale $s_j$ satisfies $2s_j\le\lambda(S)$, so the immediate second call is a split case and cannot return to local replacement. This prevents circular use of the induction hypothesis.

\begin{lemma}\label{lem:macro-expansion}
Fix an integer $m\ge2$. Assume that $\mathsf P(j)$ holds for every integer $j$ with $2\le j<m$. Let $S$ have exactly $m$ leaves and let $|S|\ge2s>0$. Then there are a finite positive Borel measure $\eta_{S,s}$ on $X_{S,s}$ and finitely many triples $(u_\beta,S_\beta,t_\beta)$, with $u_\beta>0$, such that
\begin{equation}\label{eq:macro-cover}
 c_{\eta_{S,s}}(x)+\sum_{\beta:\,x\in S_\beta}u_\beta\ge1
 \qquad(x\in S),
\end{equation}
\begin{equation}\label{eq:macro-budget}
 \pi_*\eta_{S,s}+\sum_\beta u_\beta\Lam{S_\beta}{t_\beta}
 \le\Lam{S}{s},
\end{equation}
and every retained child satisfies
\begin{equation}\label{eq:macro-child}
 |L(S_\beta)|=m,
 \qquad 0<t_\beta\le s,
 \qquad |S_\beta|\le\frac{|S|}{2},
 \qquad |S_\beta|\ge2t_\beta.
\end{equation}
For $m=2$, the conclusion holds without the induction assumption.
\end{lemma}

\begin{proof}
Let $\lambda=\lambda(S)$.

\medskip
\noindent\textit{Case 1. $\lambda\ge2s$ and $\lambda>|S|/2$.}
By \cref{prop:three-piece}, write $S=S_1\cup S_2\cup S_3$ so that each pairwise union has length at least $\lambda\ge2s$ and fewer than $m$ leaves. Each pairwise union is nontrivial and has between $2$ and $m-1$ leaves, so the strong induction hypothesis gives a fractional cover of it with its target radius measure. By \cref{prop:half-sum}, their half-sum is a fractional cover $\eta_{S,s}$ of $S$ with $\pi_*\eta_{S,s}\le\Lam{S}{s}$. No child is retained. This case cannot occur for $m=2$ by \cref{lem:interval-split}.

\medskip
\noindent\textit{Case 2. $\lambda\ge2s$ and $\lambda=|S|/2$.}
Let $S=A\cup B$ be an optimal split. Then $|A|=|B|=|S|/2\ge2s$. By \cref{lem:leaf-monotonicity}, each side has at most $m$ leaves. If a side $R\in\{A,B\}$ has fewer than $m$ leaves, then it has at least two leaves and we let $\zeta_R$ be the fractional cover given by the strong induction hypothesis. If it has $m$ leaves, then we retain the triple $(1,R,s)$. Let $\eta_{S,s}$ be the sum of the measures $\zeta_R$ over the resolved sides.

Each point of $A$ is either covered by $\zeta_A$ or counted by the retained copy of $A$, and the same holds for $B$. Thus \eqref{eq:macro-cover} holds. The budget follows from \eqref{eq:balanced-budget}. Every retained side has length $|S|/2$ and still satisfies $|R|\ge2s$. When $m=2$, \cref{lem:interval-split} shows that both sides are intervals, so no lower-leaf result is used.

\medskip
\noindent\textit{Case 3. $\lambda<2s$.}
Apply \cref{prop:local-replacement}. We obtain a certificate $(\eta_0;\{(q_j,s_j)\}_{j=1}^q)$ for $(S,s)$ with $2s_j\le\lambda$ for every remaining term. Because $2s_j\le\lambda(S)$, the same tree at scale $s_j$ satisfies Case~1 or Case~2, so this second call does not return to Case~3. Apply that case and denote the resolved measure by $\eta_j$ and the retained triples by $(u_{jh},S_{jh},t_{jh})$.

Set $\eta_{S,s}=\eta_0+\sum_{j=1}^q q_j\eta_j$, and retain $(q_j u_{jh},S_{jh},t_{jh})$. Each certificate and each split has finitely many children, so the retained family is finite. For $x\in S$,
\begin{align*}
 c_{\eta_{S,s}}(x)+\sum_{j,h:\,x\in S_{jh}}q_j u_{jh}
 &=c_{\eta_0}(x)+\sum_{j=1}^q q_j
 \left(c_{\eta_j}(x)+\sum_{h:\,x\in S_{jh}}u_{jh}\right)\\
 &\ge c_{\eta_0}(x)+\sum_{j=1}^q q_j\ge1.
\end{align*}
Also,
\begin{align*}
 \pi_*\eta_{S,s}+\sum_{j,h}q_j u_{jh}\Lam{S_{jh}}{t_{jh}}
 &\le\pi_*\eta_0+\sum_{j=1}^q q_j\Lam{S}{s_j}\\
 &\le\Lam{S}{s}.
\end{align*}
Every retained child comes from a balanced split at target scale $s$ in Case~2, or at a smaller replacement scale $s_j<s$ in Case~3. Hence $0<t_\beta\le s$, and all the assertions in \eqref{eq:macro-child} hold.
\end{proof}

Norin and Turcotte~\cite[Section~6.1]{NorinTurcotte} reported the corresponding exact conclusion for metric trees with at most three leaves. We now prove the fractional statement for arbitrary finite metric trees.

\begin{proposition}\label{prop:fractional-main}
Let $T$ be a finite metric tree of length $L\ge2r>0$. Then there is a fractional marked-ball cover $\xi$ on $X_{T,r}$ such that $\pi_*\xi\le\Lam{T}{r}$.
\end{proposition}

\begin{proof}
We extend the low-leaf range reported by Norin and Turcotte~\cite[Section~6.1]{NorinTurcotte} by proving $\mathsf P(m)$ for every $m\ge2$ through strong induction. Fix $m=|L(T)|$. The induction hypothesis is that $\mathsf P(j)$ holds for every $2\le j<m$; it is vacuous when $m=2$. Therefore \cref{lem:macro-expansion} applies to every subproblem with $m$ leaves. Notice that, when $m=2$, \cref{lem:interval-split} rules out the non-balanced three-piece case, so no lower-leaf assertion is used.

The recursive construction maintains four facts: the resolved measure and the unresolved subproblems jointly cover every point; their radius budgets sum to at most the original target; every unresolved subtree has exactly $m$ leaves; and each unresolved child has at most half the length of its parent.

At stage $n$, let $\xi_n$ be a finite positive Borel measure on $X_{T,r}$, and let $\mathcal U_n=\{(w_\alpha,T_\alpha,r_\alpha)\}_\alpha$ be a finite family of remaining subproblems. Here $w_\alpha>0$, the subtree $T_\alpha$ has $m$ leaves, and
$0<r_\alpha\le r$ with $|T_\alpha|\ge2r_\alpha$. We keep the two inequalities
\begin{equation}\label{eq:coverage-invariant}
 c_{\xi_n}(x)+\sum_{\alpha:\,x\in T_\alpha}w_\alpha\ge1
 \qquad(x\in T)
\end{equation}
and
\begin{equation}\label{eq:budget-invariant}
 \pi_*\xi_n+\sum_\alpha w_\alpha\Lam{T_\alpha}{r_\alpha}
 \le\Lam{T}{r}.
\end{equation}
Start with $\xi_0=0$ and $\mathcal U_0=\{(1,T,r)\}$.

Suppose that stage $n$ is defined. Apply \cref{lem:macro-expansion} to each $(T_\alpha,r_\alpha)$. Let its output be $\eta_\alpha$ and $(u_{\alpha\beta},S_{\alpha\beta},t_{\alpha\beta})$. View every measure on $T_\alpha$ as a measure on $T$. Since $\mathcal U_n$ is finite and every macro-expansion has finitely many children, $\mathcal U_{n+1}$ is finite. Set $\xi_{n+1}=\xi_n+\sum_\alpha w_\alpha\eta_\alpha$ and
\[
 \mathcal U_{n+1}
 =\{(w_\alpha u_{\alpha\beta},S_{\alpha\beta},t_{\alpha\beta}):\alpha,\beta\}.
\]
For $x\in T$, by \eqref{eq:macro-cover}
\begin{align*}
 c_{\xi_{n+1}}(x)
 &+\sum_{\alpha,\beta:\,x\in S_{\alpha\beta}}
 w_\alpha u_{\alpha\beta}\\
 &\ge c_{\xi_n}(x)+\sum_{\alpha:\,x\in T_\alpha}w_\alpha
 \left(c_{\eta_\alpha}(x)
 +\sum_{\beta:\,x\in S_{\alpha\beta}}u_{\alpha\beta}\right)\\
 &\ge c_{\xi_n}(x)+\sum_{\alpha:\,x\in T_\alpha}w_\alpha\ge1.
\end{align*}
Thus \eqref{eq:coverage-invariant} holds at stage $n+1$. By \eqref{eq:macro-budget},
\begin{align*}
 \pi_*\xi_{n+1}
 &+\sum_{\alpha,\beta}w_\alpha u_{\alpha\beta}
 \Lam{S_{\alpha\beta}}{t_{\alpha\beta}}\\
 &=\pi_*\xi_n+\sum_\alpha w_\alpha
 \left(\pi_*\eta_\alpha
 +\sum_\beta u_{\alpha\beta}\Lam{S_{\alpha\beta}}{t_{\alpha\beta}}\right)\\
 &\le\pi_*\xi_n+\sum_\alpha w_\alpha\Lam{T_\alpha}{r_\alpha}
 \le\Lam{T}{r}.
\end{align*}
Hence \eqref{eq:budget-invariant} also holds, and $\xi_n\le\xi_{n+1}$.

Every retained child has at most half the length of its parent. Therefore each $(w_\alpha,T_\alpha,r_\alpha)\in\mathcal U_n$ satisfies $|T_\alpha|\le L/2^n$ and $r_\alpha\le\rho_n:=L/2^{n+1}$. Put $\Gamma_n=\sum_\alpha w_\alpha\Lam{T_\alpha}{r_\alpha}$ and $W_n=\sum_\alpha w_\alpha$. The measure $\Gamma_n$ is supported on $[0,\rho_n]$, and \eqref{eq:budget-invariant} gives $\Gamma_n\le\Lam{T}{r}$. Hence $\Gamma_n(\R_{\ge0})\le\Lam{T}{r}([0,\rho_n])=(L/r^2)\min\{\rho_n,r\}$. For each remaining term, $\Lam{T_\alpha}{r_\alpha}(\R_{\ge0})=|T_\alpha|/r_\alpha\ge2$. It follows that
\begin{equation}\label{eq:weight-decay}
 2W_n\le\Gamma_n(\R_{\ge0})
 \le\frac{L}{r^2}\min\{\rho_n,r\}\longrightarrow0.
\end{equation}

The measures $\xi_n$ increase, and \eqref{eq:budget-invariant} gives $\xi_n(X_{T,r})\le L/r$. Let $\Delta_0=\xi_0$ and $\Delta_n=\xi_n-\xi_{n-1}$ for $n\ge1$. Define the finite Borel measure $\xi=\sum_{n=0}^\infty\Delta_n$. Then $\xi(A)=\lim_n\xi_n(A)$ for every Borel set $A$. For each Borel set $B\subseteq[0,r]$, we have $(\pi_*\xi)(B)=\lim_{n\to\infty}(\pi_*\xi_n)(B)\le\Lam{T}{r}(B)$. Thus $\pi_*\xi\le\Lam{T}{r}$.

Finally, \eqref{eq:coverage-invariant} gives $c_{\xi_n}(x)\ge1-W_n$ for every $x\in T$. Since $\xi_n\uparrow\xi$, monotone convergence and \eqref{eq:weight-decay} yield $c_\xi(x)=\lim_n c_{\xi_n}(x)\ge1$. Hence $\xi$ is a fractional marked-ball cover.
\end{proof}

\begin{remark}\label{rem:finite-depth}
The construction in the proof of \cref{prop:fractional-main} is quantitative. At depth $n$, the unresolved weight satisfies $W_n\le(L/(2r^2))\min\{r,L/2^{n+1}\}$. The resolved measure $\xi_n$ obeys
$\pi_*\xi_n\le\Lam{T}{r}$ and
$c_{\xi_n}(x)\ge1-W_n$ for every $x\in T$. Consequently, whenever $W_n<1$, the measure $(1-W_n)^{-1}\xi_n$ is a fractional marked-ball cover whose radius measure is at most $(1-W_n)^{-1}\Lam{T}{r}$. By \cref{cor:budgeted-rounding}, it yields a random $0$-good cover with the same bound. In particular, if $W_n\le\varepsilon/(1+\varepsilon)$, then the expected radius measure is at most $(1+\varepsilon)\Lam{T}{r}$. Thus the exact limiting argument contains an explicit geometrically convergent finite-depth approximation.
\end{remark}

\begin{proof}[Proof of \cref{thm:main}]
By \cref{prop:fractional-main}, there is a fractional marked-ball cover $\xi$ on $X_{T,r}$ with $\pi_*\xi\le\Lam{T}{r}$. By \cref{cor:budgeted-rounding}, there is a probability measure $\nu$ on $\C(T,0)$ such that $E_\nu\le\pi_*\xi\le\Lam{T}{r}=(L/r)U[0,r]$. This proves the theorem.
\end{proof}

\section{Budgeted random--fractional duality}\label{sec:duality}

Fractional matching and covering duality for infinite hypergraphs is studied by Aharoni and Holzman~\cite{AharoniHolzman}, while Rademacher, Toriello, and Vielma~\cite{RademacherTorielloVielma} treat infinite-dimensional packing and covering polyhedra. We prove the compact radius-budget form required for metric-tree balls.

The rounding theorem has a converse: a random cover with finite expected radius measure can be lifted to a fractional marked-ball cover with the same radius measure.

\begin{lemma}\label{lem:random-lifting}
Let $T$ be a finite metric tree, let $r>0$, and let $\nu$ be a probability measure on $\C(T)$ such that almost every sampled radius lies in $[0,r]$ and $E_\nu([0,r])<\infty$. Then there is a fractional marked-ball cover $\xi$ on $X_{T,r}$ such that $\pi_*\xi=E_\nu$.
\end{lemma}

\begin{proof}
On every layer $\C_m(T)\cap[0,r]^m$, use the Borel center selection in \eqref{eq:borel-lifting}. These layerwise selections define a Borel map from the full-measure set of bounded-radius covers to $\Nf(X_{T,r})$. We extend it arbitrarily to the null complement. Its push-forward is a random finite marked cover $\mathsf P$. The assumption $E_\nu([0,r])<\infty$ says that $\mathsf P$ has finite mean mass. Put $\xi=I_{\mathsf P}$. Every realization covers $T$, so $c_\xi(x)\ge1$ for every $x\in T$. For every Borel set $A\subseteq[0,r]$, by Tonelli's theorem, we have
\[
 (\pi_*\xi)(A)
 =I_{\mathsf P}(T\times A)
 =\int\bigl|\{i:r_i\in A\}\bigr|\,d\nu
 =E_\nu(A).
\]
\end{proof}

\begin{definition}
Let $T$ be a finite metric tree, let $R>0$, and let $\beta$ be a finite positive Borel measure on $[0,R]$. We call $\beta$ a \emph{radius budget}, put $X=X_{T,R}$, and write $$\mathcal K_\beta=\{\xi:\xi\text{ is a finite positive Borel measure on }X\text{ and }\pi_*\xi\le\beta\}.$$ For $x\in T$, let $F_x=\{(v,s)\in X:d_T(x,v)\le s\}$ be its closed coverage section. A finite positive Borel measure on $T$ is \emph{atomic} if it has the form $\sum_{i=1}^n a_i\delta_{x_i}$ with $a_i\ge0$. For a finite positive Borel measure $\sigma$ on $T$, its \emph{ball-concentration function} is $M_\sigma(s)=\max_{v\in T}\sigma(B_T(v,s))$ for $s\ge0$.
\end{definition}

For fixed $s$, the map $v\mapsto\sigma(B_T(v,s))$ is upper semicontinuous: if $v_n\to v$, then $B_T(v_n,s)\subseteq B_T(v,s+\varepsilon)$ for all large $n$, and continuity from above as $\varepsilon\downarrow0$ gives the claim. Hence the maximum is attained. The function $M_\sigma$ is nondecreasing and therefore Borel measurable.

\begin{lemma}\label{lem:budget-set-compact}
For every finite metric tree $T$, every $R>0$, and every radius budget $\beta$ on $[0,R]$, the set $\mathcal K_\beta$ is compact and convex in the weak topology of finite positive Borel measures on $X_{T,R}$.
\end{lemma}

\begin{proof}
Every $\xi\in\mathcal K_\beta$ has mass at most $\beta([0,R])$. Since $X_{T,R}$ is compact, this uniform mass bound gives weak relative compactness. To see that the set is closed, let $\xi_n\to\xi$ weakly with $\pi_*\xi_n\le\beta$. For every nonnegative $f\in C([0,R])$, we have $\int f\,d(\pi_*\xi)=\lim_{n\to\infty}\int f\,d(\pi_*\xi_n)\le\int f\,d\beta$. By \eqref{eq:measure-order-standard}, $\pi_*\xi\le\beta$. Convexity is immediate.
\end{proof}

\begin{lemma}\label{lem:atomic-support-function}
Let $T$ be a finite metric tree, let $R>0$, let $\beta$ be a radius budget on $[0,R]$, and let $\sigma=\sum_{i=1}^n a_i\delta_{x_i}$ be a finite positive atomic measure on $T$. Put $h_\sigma(v,s)=\sigma(B_T(v,s))$. Then
\begin{equation}\label{eq:support-function}
 \sup_{\xi\in\mathcal K_\beta}\int_{X_{T,R}}h_\sigma\,d\xi
 =\int_{[0,R]}M_\sigma(s)\,d\beta(s).
\end{equation}
\end{lemma}

\begin{proof}
The upper bound follows from $h_\sigma(v,s)\le M_\sigma(s)$ and $\pi_*\xi\le\beta$. If $\sigma=0$, equality is immediate. Assume henceforth that $\sigma(T)>0$. For the reverse inequality, for every $J\subseteq[n]$ put $q_J=\sum_{j\in J}a_j$ and $\rho_J=\min_{v\in T}\max_{j\in J}d_T(v,x_j)$, with $q_\varnothing=\rho_\varnothing=0$, and choose a center $c_J$ attaining $\rho_J$ for each nonempty $J$. A ball of radius $s$ determines a set $J$ with $\rho_J\le s$, while a center attaining $\rho_J$ gives a radius-$s$ ball containing all atoms indexed by $J$. Consequently, $M_\sigma(s)=\max\{q_J:\rho_J\le s\}$.
Choose the first maximizing set $J(s)$ in a fixed ordering of the finitely many subsets of $[n]$. This selector is Borel because it changes only at the finitely many values $\rho_J$. Put $c(s)=c_{J(s)}$. The ball $B_T(c(s),s)$ contains all atoms indexed by $J(s)$, so
$h_\sigma(c(s),s)\ge q_{J(s)}=M_\sigma(s)$; the reverse inequality is the definition of $M_\sigma(s)$. Thus equality holds. The push-forward of $\beta$ under $s\mapsto(c(s),s)$ belongs to $\mathcal K_\beta$ and attains the right side of \eqref{eq:support-function}.
\end{proof}

\begin{lemma}\label{lem:finite-demand-feasibility}
Let $T$ be a finite metric tree, let $R>0$, and let $\beta$ be a radius budget on $[0,R]$. Assume that every finite positive atomic measure $\sigma$ on $T$ satisfies
\begin{equation}\label{eq:budgeted-dual-concentration}
 \sigma(T)\le\int_{[0,R]}M_\sigma(s)\,d\beta(s).
\end{equation}
Then, for every finite set $D\subseteq T$, some $\xi\in\mathcal K_\beta$ satisfies $\xi(F_x)\ge1$ for every $x\in D$.
\end{lemma}

\begin{proof}
Write $D=\{x_1,\ldots,x_n\}$ and define
\[
 C_D=\left\{z\in\R_{\ge0}^n:
 \text{there is }\xi\in\mathcal K_\beta\text{ with }
 z_i\le\xi(F_{x_i})\text{ for every }i\right\}.
\]
This set is convex, bounded, and coordinatewise downward closed. It is also closed. Indeed, if $z^{(h)}\to z$ and $z_i^{(h)}\le\xi_h(F_{x_i})$, take a weakly convergent subsequence in the compact set $\mathcal K_\beta$. Since each $F_{x_i}$ is closed, by the portmanteau theorem, we have $z_i\le\limsup_h\xi_h(F_{x_i})\le\xi(F_{x_i})$.

Suppose that $\mathbf1\notin C_D$. Strict separation gives $a\in\R^n$ such that
\begin{equation}\label{eq:strict-separation}
 \sum_{i=1}^na_i>
 \sup_{z\in C_D}\sum_{i=1}^na_i z_i.
\end{equation}
Because $C_D\subseteq\R_{\ge0}^n$ is coordinatewise downward closed, setting a coordinate of $z$ to zero keeps the vector in $C_D$. Replacing $a$ by its positive part therefore leaves the supremum on the right unchanged and cannot decrease the left side. We may assume $a_i\ge0$. For $\sigma=\sum_i a_i\delta_{x_i}$, nonnegativity of $a$ and the definition of $C_D$ give
\[
 \sup_{z\in C_D}\sum_i a_i z_i
 =\sup_{\xi\in\mathcal K_\beta}\sum_i a_i\xi(F_{x_i})
 =\sup_{\xi\in\mathcal K_\beta}\int h_\sigma\,d\xi.
\]
By \cref{lem:atomic-support-function}, the last expression is
$\int M_\sigma\,d\beta$. Thus \eqref{eq:strict-separation} contradicts \eqref{eq:budgeted-dual-concentration}. Hence $\mathbf1\in C_D$, which is the required finite-demand cover.
\end{proof}

\begin{proposition}\label{prop:fractional-duality}
Let $T$ be a nontrivial finite metric tree, let $R>0$, and let $\beta$ be a finite positive Borel measure on $[0,R]$. The following statements are equivalent.
\begin{enumerate}[label=\textup{(\roman*)}]
\item There is a fractional marked-ball cover $\xi$ on $X_{T,R}$ such that $\pi_*\xi\le\beta$.
\item Every finite positive Borel measure $\sigma$ on $T$ satisfies \eqref{eq:budgeted-dual-concentration}.
\item Every finite positive atomic measure $\sigma$ on $T$ satisfies \eqref{eq:budgeted-dual-concentration}.
\end{enumerate}
\end{proposition}

\begin{proof}
We use the separation and finite-intersection scheme standard in fractional covering duality; see Aharoni and Holzman~\cite{AharoniHolzman}. Assume first that (i) holds. Since $c_\xi(x)\ge1$ for every $x\in T$, by  Tonelli's theorem, we have for every finite positive Borel measure $\sigma$ on $T$,
\begin{align*}
 \sigma(T)
 &\le\int_Tc_\xi(x)\,d\sigma(x)\\
 &=\int_{X_{T,R}}\sigma(B_T(v,s))\,d\xi(v,s)\\
 &\le\int_{[0,R]}M_\sigma(s)\,d(\pi_*\xi)(s)\\
 &\le\int_{[0,R]}M_\sigma(s)\,d\beta(s).
\end{align*}
Thus (i) implies (ii), and (ii) implies (iii).

Assume (iii). For every $x\in T$, let $\mathcal A_x=\{\xi\in\mathcal K_\beta:\xi(F_x)\ge1\}$. The set $\mathcal A_x$ is closed because $F_x$ is closed and $\xi\mapsto\xi(F_x)$ is upper semicontinuous. By \cref{lem:finite-demand-feasibility}, the family $(\mathcal A_x)_{x\in T}$ has the finite intersection property. The compactness in \cref{lem:budget-set-compact} gives a measure in $\bigcap_{x\in T}\mathcal A_x$. It is a fractional marked-ball cover with radius measure at most $\beta$, proving (i).
\end{proof}

Combining lifting, compact rounding, and strengthened trimming gives an exact theorem for arbitrary radius budgets.

\begin{theorem}\label{thm:budgeted-duality}
Let $T$ be a nontrivial finite metric tree of length $L$, let $0<R\le L$, and let $\beta$ be a finite positive Borel measure on $[0,R]$, extended by zero to $\R_{\ge0}$. The following statements are equivalent.
\begin{enumerate}[label=\textup{(\roman*)}]
\item There is a probability measure $\nu$ on $\C(T,0)$ such that $E_\nu\le\beta$.
\item There is a fractional marked-ball cover $\xi$ on $X_{T,R}$ such that $\pi_*\xi\le\beta$.
\item Every finite positive Borel measure $\sigma$ on $T$ satisfies \eqref{eq:budgeted-dual-concentration}.
\end{enumerate}
It is enough in (iii) to test finite positive atomic measures.
\end{theorem}

\begin{proof}
If (i) holds, domination by $\beta$ forces almost every sampled radius to lie in $[0,R]$ and gives finite expected cover size. Apply \cref{lem:random-lifting} to obtain (ii). If (ii) holds, \cref{cor:budgeted-rounding} gives (i), because $R\le L$. The equivalence of (ii) and (iii), as well as the final assertion, is \cref{prop:fractional-duality}.
\end{proof}

For the uniform target, the budgeted theorem becomes the following intrinsic characterization.

\begin{corollary}\label{thm:random-fractional-duality}
Let $T$ be a nontrivial finite metric tree of length $L$, and let $r>0$. The following statements are equivalent.
\begin{enumerate}[label=\textup{(\roman*)}]
\item There is a probability measure $\nu$ on $\C(T,0)$ such that $E_\nu\le\Lam{T}{r}$.
\item There is a fractional marked-ball cover $\xi$ on $X_{T,r}$ such that $\pi_*\xi\le\Lam{T}{r}$.
\item Every finite positive Borel measure $\sigma$ on $T$ satisfies
\begin{equation}\label{eq:dual-concentration}
 \sigma(T)\le\frac{L}{r^2}\int_0^rM_\sigma(s)\,ds.
\end{equation}
\end{enumerate}
It is enough in (iii) to test finite positive atomic measures.
\end{corollary}

\begin{proof}
For $0<r\le L$, this is \cref{thm:budgeted-duality} with $R=r$ and $\beta=\Lam{T}{r}$. If $r>L$, the target has total mass $L/r<1$, so neither a random nor a fractional cover can be dominated by it; the concentration inequality also fails for a unit point mass.
\end{proof}

\section{Consequences, rigidity, and sharp extensions}\label{sec:consequences}

The uniform criterion \eqref{eq:dual-concentration} first gives volume and packing obstructions. We then use the metric radius to extend the existence range, determine the exact threshold for equal-arm stars, and finish with deterministic extraction and stability.

\subsection{Volume and packing obstructions}

Let $\ell_T$ denote length measure on $T$, and define the maximal ball-volume profile by $V_T(s)=\max_{v\in T}\ell_T(B_T(v,s))$.
Since $V_T=M_{\ell_T}$, the measurability observation preceding \cref{lem:budget-set-compact} shows that $V_T$ is nondecreasing and Borel measurable.

\begin{corollary}\label{cor:budget-obstructions}
Let $T$ be a nontrivial finite metric tree of length $L$, let $0<R\le L$, and let $\beta$ be a radius budget satisfying the equivalent conditions of \cref{thm:budgeted-duality}.
\begin{enumerate}[label=\textup{(\roman*)}]
\item The maximal ball-volume profile satisfies
\begin{equation}\label{eq:budget-volume}
 L\le\int_{[0,R]}V_T(s)\,d\beta(s).
\end{equation}
\item If $T$ contains $m\ge2$ points whose pairwise distances are at least $\delta>0$, then
\begin{equation}\label{eq:budget-packing}
 m\le\beta\bigl([0,R]\cap[0,\delta/2)\bigr)
   +m\,\beta\bigl([0,R]\cap[\delta/2,\infty)\bigr).
\end{equation}
\end{enumerate}
\end{corollary}

\begin{proof}
Apply the concentration criterion in \cref{thm:budgeted-duality} first to length measure $\ell_T$, for which $M_{\ell_T}=V_T$, and then to the sum of unit point masses on the separated set. A ball of radius less than $\delta/2$ contains at most one of the selected points, while every ball contains at most $m$ of them.
\end{proof}

\begin{corollary}\label{cor:volume-profile}
If the equivalent conditions of \cref{thm:random-fractional-duality} hold, then $r^2\le\int_0^rV_T(s)\,ds$.
\end{corollary}

\begin{proof}
Apply \eqref{eq:budget-volume} with $R=r$ and $\beta=\Lam{T}{r}$, and cancel $L>0$.
\end{proof}

\begin{corollary}\label{cor:packing-obstruction}
Suppose that $T$ contains $m\ge2$ points whose pairwise distances are at least $\delta>0$. If the equivalent conditions of \cref{thm:random-fractional-duality} hold, then
\begin{enumerate}[label=\textup{(\roman*)}]
\item if $r\le\delta/2$, then $r\le L/m$;
\item if $r>\delta/2$, then
\begin{equation}\label{eq:packing-obstruction}
 r^2\le L\left(r-\frac{\delta}{2}\left(1-\frac1m\right)\right).
\end{equation}
\end{enumerate}
\end{corollary}

\begin{proof}
Apply \eqref{eq:budget-packing} with $R=r$ and $\beta=\Lam{T}{r}$. If $r\le\delta/2$, it gives $m\le L/r$, which is (i). If $r>\delta/2$, it gives $m\le(L/r^2)(\delta/2+m(r-\delta/2))$. After division by $m$, this is \eqref{eq:packing-obstruction}.
\end{proof}

\subsection{A diameter-dependent extension}

Recall from \cref{lem:tree-radius} that a nontrivial finite metric tree of diameter $D$ has metric radius $D/2$.

\begin{theorem}\label{thm:diameter-extension}
Let $T$ be a nontrivial finite metric tree of length $L$, diameter $D$, and metric radius $\rho=D/2$. For every $r$ satisfying
\begin{equation}\label{eq:diameter-range}
 0<r\le L-\rho=L-\frac D2,
\end{equation}
there is a probability measure $\nu$ on $\C(T,0)$ such that $E_\nu\le\Lam{T}{r}$.
\end{theorem}

\begin{proof}
By \cref{lem:tree-radius}, $\rho=D/2$. If $r\le L/2$,  then we apply \cref{thm:main}. Suppose that
$L/2<r\le L-\rho$. Then $\rho<L/2<r$. By \cref{prop:fractional-main}, applied at scale $\rho$, there is a fractional marked-ball cover $\xi_\rho$ on $X_{T,\rho}$ such that $\pi_*\xi_\rho\le(L/\rho^2)\one_{[0,\rho]}(s)\,ds$. Choose a center $c$ of $T$, so $B_T(c,s)=T$ for every $s\ge\rho$. For every continuous $f$, define a measure $\zeta$ on $X_{T,r}$ by $\int f\,d\zeta=(L/r^2)\int_\rho^r f(c,s)\,ds$, and put $\xi_r=(\rho^2/r^2)\xi_\rho+\zeta$.
Its radius measure satisfies
\[
 \pi_*\xi_r
 \le\frac{L}{r^2}\one_{[0,\rho]}(s)\,ds
   +\frac{L}{r^2}\one_{[\rho,r]}(s)\,ds
 =\Lam{T}{r}.
\]
For every $x\in T$, $c_{\xi_r}(x)\ge(\rho^2+L(r-\rho))/r^2$. The numerator is at least $r^2$ because $\rho^2+L(r-\rho)-r^2=(r-\rho)(L-\rho-r)\ge0$. Thus $\xi_r$ is a fractional marked-ball cover. Since
$r\le L-\rho<L$, \cref{cor:budgeted-rounding} gives the required probability measure.
\end{proof}

\begin{corollary}\label{cor:nonpath-extension}
If $T$ is not a metric interval, then $\diamT(T)<L$, and the range in \eqref{eq:diameter-range} strictly contains $0<r\le L/2$.
\end{corollary}

\begin{proof}
A diameter arc has length $D$. If $D=L$, no positive-length edge lies outside that arc, so $T$ is an interval. Hence a non-interval tree has $D<L$, and therefore $L-D/2>L/2$.
\end{proof}

We next identify when one ball in every sample realizes the target.

\begin{proposition}\label{prop:one-ball-extension}
Let $T$ be a nontrivial finite metric tree of length $L$ and metric radius $\rho$, and let $r>0$. There is a probability measure supported on one-ball covers in $\C(T,0)$ whose expected radius measure is at most $\Lam{T}{r}$ if and only if
\begin{equation}\label{eq:one-ball-condition}
 r^2\le L(r-\rho).
\end{equation}
\end{proposition}

\begin{proof}
Suppose first that such a probability measure exists. A ball covering $T$ has radius at least $\rho$, while the measure domination forces the sampled radius to lie in $[0,r]$ almost surely. Hence $1=E_\nu([\rho,r])\le L(r-\rho)/r^2$, which is \eqref{eq:one-ball-condition}.

Conversely, assume \eqref{eq:one-ball-condition}. It implies $r>\rho$ and, since $\rho>0$, also $r<L$. Choose a center $c$ of $T$ and let $R$ be uniform on $[\rho,r]$. Then $B_T(c,R)=T$ and $(R)\in\C(T,0)$ for every realization. Moreover, $U[\rho,r]\le(L/r)U[0,r]$ because $1/(r-\rho)\le L/r^2$.
\end{proof}

\subsection{The exact threshold for equal-arm stars}

Let $S_{k,a}$ denote the metric star with $k\ge2$ arms, each of length $a>0$. Thus $|S_{k,a}|=ka$ and $\rho(S_{k,a})=a$.

\begin{theorem}\label{thm:star-threshold}
Let $k\ge2$, $a>0$, and $r>0$. There is a probability measure $\nu$ on $\C(S_{k,a},0)$ such that $E_\nu\le(ka/r)U[0,r]$ if and only if $0<r\le(k-1)a$.
\end{theorem}

\begin{proof}
If $r\le(k-1)a=|S_{k,a}|-\rho(S_{k,a})$, sufficiency follows from \cref{thm:diameter-extension}.

For necessity, let $z_1,\ldots,z_k$ be the endpoints of the arms. A ball of radius $s<a$ contains at most one endpoint, because distinct endpoints are at distance $2a$. A ball of radius $s\ge a$ contains at most $k$ endpoints. Define
\[
 q(s)=
 \begin{cases}
 1,&0\le s<a,\\
 k,&s\ge a.
 \end{cases}
\]
Every ball cover $(r_1,\ldots,r_m)$ satisfies
$k\le\sum_iq(r_i)$. Therefore $k\le\int q\,dE_\nu\le(ka/r^2)\int_0^r q(s)\,ds$. If $r\le a$, then $r\le(k-1)a$ already. If $r>a$, the last inequality becomes $r^2-kar+(k-1)a^2\le0$, or equivalently $(r-a)(r-(k-1)a)\le0$. Since $r>a$, we obtain $r\le(k-1)a$.
\end{proof}

\begin{remark}
For $k=2$, \cref{thm:star-threshold} is the interval threshold. For every $k\ge3$, it shows that the diameter range in \cref{thm:diameter-extension} is exact on $S_{k,a}$. Thus the extension is not only strict for non-path trees; it is best possible for an infinite family with arbitrarily many leaves.
\end{remark}

\subsection{Deterministic cost bounds}

\begin{proposition}\label{prop:cost-extraction}
Let $T$ be a finite metric tree, let $R>0$, and let $\beta$ be a finite positive Borel measure on $[0,R]$. Suppose that a probability measure $\nu$ on $\C(T,0)$ satisfies $E_\nu\le\beta$. If $\phi:[0,R]\to[0,\infty)$ is Borel and $\beta$-integrable, then some cover $(r_1,\ldots,r_m)\in\C(T,0)$ satisfies $\sum_{i=1}^m\phi(r_i)\le\int_{[0,R]}\phi(s)\,d\beta(s)$. In particular, for the uniform budget $\beta=\Lam{T}{r}$ this becomes $\sum_{i=1}^m\phi(r_i)\le(L/r^2)\int_0^r\phi(s)\,ds$.
\end{proposition}

\begin{proof}
Domination by $\beta$ forces all sampled radii to lie in $[0,R]$ almost surely. By Tonelli's theorem, $\E_\nu[\sum_i\phi(r_i)]=\int_{[0,R]}\phi\,dE_\nu\le\int_{[0,R]}\phi\,d\beta$. The set of sampled covers satisfying the same upper bound has positive $\nu$-measure; otherwise the displayed random variable would be strictly larger than its upper bound almost surely. In particular, such a cover exists.
\end{proof}

\begin{corollary}\label{cor:deterministic-bounds}
Let $T$ be a nontrivial finite metric tree of length $L$ and metric radius $\rho$, and let $0<r\le L-\rho$.
\begin{enumerate}[label=\textup{(\roman*)}]
\item There is a $0$-good cover of $T$ using at most $\lfloor L/r\rfloor$ balls.
\item For every $p>0$, there is a $0$-good cover $(r_1,\ldots,r_m)$ satisfying $\sum_{i=1}^m r_i^p\le Lr^{p-1}/(p+1)$. In particular, one may require $\sum_i r_i\le L/2$.
\item For every fixed $t\in[0,r]$, there is a $0$-good cover, depending on $t$, such that $|\{i:r_i\ge t\}|\le\lfloor L(r-t)/r^2\rfloor$.
\end{enumerate}
\end{corollary}

\begin{proof}
Use \cref{thm:diameter-extension} and apply \cref{prop:cost-extraction} with $\phi=1$, with $\phi(s)=s^p$, and with $\phi=\one_{[t,r]}$, respectively. In the first and third cases, the left side is an integer.
\end{proof}

\subsection{Diameter-defect stability and interval rigidity}

\begin{lemma}\label{lem:path-ball-length}
Let $P=T[a,b]$ be a path of length $D$, let $\ell_P$ be its length measure, and define $g_P(v,s)=\ell_P(B_T(v,s)\cap P)$. Then $g_P:T\times\R_{\ge0}\to\R_{\ge0}$ is continuous.
\end{lemma}

\begin{proof}
Let $p(v)$ be the nearest-point projection of $v$ onto the closed convex subtree $P$. The projection is $1$-Lipschitz: the unique-path property of a tree gives
$d_T(p(v),p(w))\le d_T(v,w)$. Hence $h(v)=d_T(v,P)$ and $\tau(v)=d_T(a,p(v))$ are continuous. For $x\in P$, the path from $v$ to $x$ passes through $p(v)$, so
$d_T(v,x)=h(v)+d_P(p(v),x)$. Put $q(v,s)=\max\{s-h(v),0\}$. Identifying $P$ with $[0,D]$ through distance from $a$, we obtain
\[
 g_P(v,s)
 =\min\{D,\tau(v)+q(v,s)\}
  -\max\{0,\tau(v)-q(v,s)\},
\]
which is continuous.
\end{proof}

\begin{theorem}\label{thm:diameter-stability}
Let $T$ be a nontrivial finite metric tree of length $L$ and diameter $D$, let $r>0$, and let $\nu$ be a probability measure on $\C(T)$ such that
$E_\nu\le\Lam{T}{r}$. Put $\Delta=\Lam{T}{r}-E_\nu$. Then
\begin{equation}\label{eq:diameter-defect-moment}
 0\le\int_0^r s\,d\Delta(s)\le\frac{L-D}{2},
\end{equation}
and, for every $t\in(0,r]$,
\begin{equation}\label{eq:diameter-defect-tail}
 \Delta([t,r])\le\frac{L-D}{2t}.
\end{equation}
Moreover, fix a diameter path $P$ and select realizing centers by the layerwise Borel lifting in \eqref{eq:borel-lifting}. For a sampled cover, put
$J_i=B_T(v_i,r_i)\cap P$. Then $A=\sum_i(2r_i-|J_i|)\ge0$ and $O=\sum_i|J_i|-D\ge0$, and
\begin{equation}\label{eq:geometric-defect}
 \E_\nu[A+O]
 =L-D-2\int_0^r s\,d\Delta(s)
 \le L-D.
\end{equation}
Consequently, for every $u>0$,
\begin{equation}\label{eq:geometric-defect-tail}
 \nu\bigl(A+O\ge u\bigr)\le\frac{L-D}{u}.
\end{equation}
\end{theorem}

\begin{proof}
The domination implies that all sampled radii lie in $[0,r]$ almost surely. Every ball meets the diameter path $P$ in an interval of length at most twice its radius. Since the sampled balls cover $P$, $D\le\sum_i|J_i|\le2\sum_i r_i$. Taking expectations, we have $D/2\le\int_0^r s\,dE_\nu(s)\le(L/r^2)\int_0^r s\,ds=L/2$.
Subtracting from the first moment of $\Lam{T}{r}$ proves \eqref{eq:diameter-defect-moment}. Since $\Delta$ is positive, $t\Delta([t,r])\le\int_t^r s\,d\Delta(s)$, which gives \eqref{eq:diameter-defect-tail}.

By \cref{lem:path-ball-length} and the layerwise Borel center selection, $A$ and $O$ are Borel random variables. They are nonnegative by the two inequalities above, and $A+O=2\sum_i r_i-D$. Taking expectations and using
$\int_0^r s\,dE_\nu(s)=L/2-\int_0^r s\,d\Delta(s)$ proves the identity in \eqref{eq:geometric-defect}. Its inequality follows from the positivity of $\Delta$, and \eqref{eq:geometric-defect-tail} follows from Markov's inequality.
\end{proof}

We now specialize \cref{thm:diameter-stability} to the zero-defect case.

\begin{theorem}\label{thm:interval-rigidity}
Let $I$ be a metric interval of length $L>0$, let $r>0$, and let $\nu$ be a probability measure on $\C(I)$ such that $E_\nu\le(L/r)U[0,r]$. Then the following statements hold.
\begin{enumerate}[label=\textup{(\roman*)}]
\item The measure inequality is an equality: $E_\nu=(L/r)U[0,r]$.
\item For $\nu$-almost every cover $(r_1,\ldots,r_m)$, all radii are positive and $\sum_{i=1}^m r_i=L/2$.
\item For $\nu$-almost every radius tuple satisfying \textup{(ii)}, every choice of centers that realizes the tuple gives intervals $B_I(v_i,r_i)$ with pairwise disjoint interiors; each interval has length $2r_i$, and their union is $I$.
\end{enumerate}
\end{theorem}

\begin{proof}
Here $D=L$. Every sampled cover satisfies
$L\le2\sum_i r_i$, while its expected radius sum is at most $L/2$. Hence
$\sum_i r_i=L/2$ almost surely.

Let $\Delta=(L/r)U[0,r]-E_\nu$. Then
$\int s\,d\Delta(s)=0$, so $\Delta$ is supported on $\{0\}$. The target measure has no atom at $0$, and therefore $\Delta=0$. This proves (i). It also gives $E_\nu(\{0\})=0$, which proves the positivity assertion in (ii).

Fix a cover and realizing centers for which (ii) holds, and put
$J_i=B_I(v_i,r_i)$. Then $L=|\bigcup_iJ_i|\le\sum_i|J_i|\le2\sum_i r_i=L$. Equality holds throughout. Thus $|J_i|=2r_i$ for every $i$, and equality between the length of the union and the sum of the lengths implies that the interiors are pairwise disjoint. This proves (iii).
\end{proof}

The following threshold is implicit in the interval obstruction and the verification for trees with at most three leaves in Norin and Turcotte~\cite[Section~6.1]{NorinTurcotte}; \cref{thm:interval-rigidity} gives the stronger equality structure.

\begin{corollary}[Norin--Turcotte, Section~6.1~\cite{NorinTurcotte}]\label{cor:interval-threshold}
Let $I$ be a metric interval of length $L>0$, and let $r>0$. There is a probability measure $\nu$ on $\C(I,0)$ such that $E_\nu\le(L/r)U[0,r]$ if and only if $L\ge2r$.
\end{corollary}

\begin{corollary}\label{cor:admissible-radii}
For a nontrivial finite metric tree $T$ of length $L$ and diameter $D$, define
\[
 \mathcal A(T)=\left\{r>0:\text{there is a probability measure }\nu
 \text{ on }\C(T,0)\text{ with }E_\nu\le\Lam{T}{r}\right\}.
\]
Then $(0,L-D/2]\subseteq\mathcal A(T)\subseteq(0,L)$. Moreover, $\mathcal A(I_L)=(0,L/2]$ for a metric interval of length $L$, and $\mathcal A(S_{k,a})=(0,(k-1)a]$ for every equal-arm metric star $S_{k,a}$.
\end{corollary}

\begin{proof}
The first inclusion is \cref{thm:diameter-extension}. If $r>L$, then $\Lam{T}{r}$ has total mass $L/r<1$, whereas the expected radius measure of every random cover has total mass at least one.

We finally exclude $r=L$. Suppose that $E_\nu\le\Lam{T}{L}=U[0,L]$. Since every sampled cover is nonempty, $1\le E_\nu([0,L])\le U[0,L]([0,L])=1$. Hence almost every sampled cover consists of exactly one ball, and the positive measure $U[0,L]-E_\nu$ has total mass zero. Thus $E_\nu=U[0,L]$. On the other hand, a one-ball cover of $T$ has radius at least $\rho(T)>0$, so $E_\nu([0,\rho(T)))=0$, whereas $U[0,L]([0,\rho(T)))=\rho(T)/L>0$, a contradiction. Therefore $\mathcal A(T)\subseteq(0,L)$. The two exact descriptions are \cref{cor:interval-threshold,thm:star-threshold}.
\end{proof}

\section*{Declaration of competing interest}

We declare that we have no known competing financial interests or personal relationships that could have appeared to influence the work reported in this paper.

\section*{Data availability}

We did not use or generate data for this study.

\section*{Declaration on the use of AI}

During the preparation of this work, we used ChatGPT 5.6 Pro to discuss possible proof strategies, audit intermediate arguments, and improve the exposition. We reviewed and verified every suggestion used in the manuscript, revised the text where necessary, and accept full responsibility for the mathematical content, citations, and final text.

\end{document}